\documentclass{siamart220329}

\usepackage{fullpage,amsmath,amsfonts,amssymb,mathrsfs,hyperref}
\usepackage{mathtools}
\usepackage{makecell}
\usepackage{todonotes}
\usepackage{xcolor}
\usepackage{tikz}
\usepackage{tikz-cd}
\usepackage{pgf}
\usepackage{fancyvrb}   
\usepackage{booktabs}

\numberwithin{equation}{section}
\definecolor{forestgreen(traditional)}{rgb}{0.0, 0.27, 0.13}
\definecolor{forestgreen(web)}{rgb}{0.13, 0.55, 0.13}
\definecolor{airforceblue}{rgb}{0.36, 0.54, 0.66}
\definecolor{bleudefrance}{rgb}{0.19, 0.55, 0.91}
\definecolor{darkorchid}{rgb}{0.6, 0.2, 0.8}
\definecolor{darkorange}{rgb}{1.0, 0.55, 0.0}
\definecolor{darkspringgreen}{rgb}{0.09, 0.45, 0.27}

\hypersetup{
  colorlinks,
  citecolor=forestgreen(traditional),
  linkcolor=magenta,
  urlcolor=airforceblue}

\def\CC{{\mathbb C}}

\def\P{{\mathbb P}}

\def\QQ{{\mathbb Q}}
\def\RR{{\mathbb R}}
\def\R{{\mathbb R}}
\def\ZZ{{\mathbb Z}}

\def\kk{{\Bbbk}}
\def\LT {{\mathrm{LT}}}

\def \sl{\mathfrak{sl}}

\DeclareMathOperator{\Supp}{Supp} 

\DeclareMathOperator{\rank}{rank}

\def\ds{\displaystyle}

\newcommand \ti[1]{\textit{#1}}

\newsiamremark{remark}{Remark}
\newsiamremark{example}{Example}

\newcommand\HF{\mathrm{HF}}
\begin{document}
\title{A new bound for the Waring rank of monomials
\thanks{\funding{The first author was supported by the National Research Foundation of Korea (NRF) grant funded by the Korea government (MSIT) (no. 2017R1E1A1A03070765) 
and the second author was supported by a KIAS Individual Grant (MG083101) at Korea Institute of Advanced Study.}}}

\author{Kangjin Han \thanks{School of Undergraduate Studies,
Daegu-Gyeongbuk Institute of Science \& Technology (DGIST),
Daegu 42988,
Republic of Korea(\email{kjhan@dgist.ac.kr}).}
\and Hyunsuk Moon\thanks{Korea Institute for Advanced Study (KIAS), Seoul 02455, Republic of Korea
(\email{hsmoon87@kias.re.kr, mhs@kaist.ac.kr})}}
\maketitle


\begin{abstract}
In this paper we consider the Waring rank of monomials over the real and the rational numbers. We give a new upper bound for it by establishing a way in which one can take a structured apolar set for any given monomial $X_0^{a_0}X_1^{a_1}\cdots X_n^{a_n}$ ($a_i>0$). This bound coincides with the real Waring rank in the case $n=1$ and in the case $\min(a_i)=1$, which are all the known cases for the real rank of monomials. Our bound is also lower than any other known general bounds for the real Waring rank. Since all of the constructions are still valid over the rational numbers, this provides a new result for the rational Waring rank of any monomial as well. Some examples and computational implementation for potential use are presented in the end.
\end{abstract}

\begin{MSCcode}
14P99, 12D05, 13P10, 14A25, 15A21, 14N15
\end{MSCcode}
\tableofcontents \setcounter{page}{1}

\section{Introduction}\label{intro}

A Waring decomposition of a polynomial is an expression of the polynomial as a sum of powers of linear forms. We study the rank of monomials over the reals and the rationals and we give an improved upper bound for it. 

Let $\kk$ be a field, $S=\kk[X_0,\ldots ,X_n]=\bigoplus_{d\ge0} S_d$ be the ring of polynomials over $\kk$ and $S_d$ be the $\kk$-vector space of homogeneous polynomials (or forms) of degree $d$.
For a given $F \in S_d$, we define a Waring decomposition of $F$ over $\kk$ as a sum 
\begin{equation}\label{waring_decomp}
F=\sum_{i=1}^{r} \lambda_i L_i^d~,
\end{equation}
where $\lambda_i \in \kk$ and $L_i$ is a linear form over $\kk$. The smallest number $r$ for which such a decomposition exists is called \textit{Waring rank of $F$ over $\kk$} and we denote it by $\rank_\kk(F)$.

Earlier studies of Waring decomposition and Waring rank, initiated by works of Sylvester and others, go back to the 19th century (see \cite{IK} for a historical background). But, despite their long history, the Waring ranks for \textit{general} forms over the complex numbers, a long-standing conjecture in this field, were determined only in the 1990s by \cite{AH} and the complex Waring rank of monomials, a specific type of polynomial, has been found in recently \cite{BBT, CCG}.
 
In general, it is known that it is very difficult to compute the rank of a form except some known cases, even though some algorithms have been proposed (see e.g. \cite{BGI} and references therein).

The Waring rank over the \textit{real} numbers is even more difficult to compute. For instance, the real Waring rank of a monomial  $X_0^{a_0}X_1^{a_1}\cdots X_n^{a_n}$ with $a_i>0$ is its degree when $n=1$ \cite{BCG} and $\frac{1}{2}\prod_{i=0}^n(a_i+1)$ when $\min(a_i)=1$ \cite{CKOV}. In \cite{CKOV}, there is also an upper bound for the real rank of any monomial $\frac{1}{2a_j}\prod_{i=0}^n(a_i+a_j)$ where $a_j$ is $\min(a_i)$, but it is not tight in general.

We give a new upper bound for the real Waring rank of any monomial as follows:

\begin{theorem}\label{main_thm}
The real Waring rank of a monomial $X_0^{a_0}X_1^{a_1}\ldots X_n^{a_n}$ with each $a_i>0$ is at most
$$ \frac{1}{2}\bigg\{\prod_{i=0}^n(a_i+1)-\prod_{i=0}^n(a_i-1)\bigg\}~.$$
\end{theorem}

We would like to note that this bound contains the two known cases of real rank, when $n=1$ and one $a_i$ is $1$, and is lower than the general bounds of \cite{CKOV} and of \cite{BT15} (see the discussion in Remark \ref{better bound} and Table \ref{compareTable}). Moreover, the same bound still holds for the rank over $\mathbb{Q}$, the field of rational numbers (see Remark \ref{Qrank}).

The paper is structured as follows. In Section \ref{sect_prelim}, we recall some preliminaries on the subject and introduce our notation. Our basic strategy to obtain the upper bound is to use the Apolarity Lemma. But, to invoke the Apolarity Lemma, one should show that a given set of equations vanishing on  the desired set of points indeed defines the points \ti{ideal-theoretically}. This is in general quite difficult except in some immediate cases such as complete intersection. That is the reason why we have a rather lengthy Section \ref{sect_upp_bd}. In that section, we prove our main result by introducing a family of real hyperbolic polynomials and investigating structures of the ideal generated by those polynomials step by step. Finally, we provide a \textsc{Macaulay2} code for computing a real apolar set and a real Waring decomposition of any given monomial via the method developed in this article as a supplementary file.

\bigskip

\paragraph*{\textbf{Acknowledgements}} The first author is grateful to J. Buczy\'nski and the Institute of Mathematics
Polish Academy of Sciences in Warsaw for inviting him to the workshop `Varieties and Group actions' in 2018 where this project was initiated. He would also like to thank Bernd Sturmfels for his stimulating question at tea time. Both authors are grateful to referees and editors for comments and suggestions to improve the exposition. 

\section{Preliminaries and notations}\label{sect_prelim}

From now on, let $\kk$ be the field of real numbers $\R$ and $S=\R[X_0,\ldots,X_n]$. Let $T=\R[x_0,\ldots,x_n]$ be the dual ring of $S$ which acts on $S$ by the differentiation $x_i\circ X_j:=\partial_i(X_j)$. For any $F\in S_d$ (the degree $d$ piece of $S$), $F^\perp=\{ \Psi\in T : \Psi\circ F=0\}$ is a homogeneous ideal in $T$. We call this $F^\perp$ the \textit{apolar ideal} (or \textit{annihilating ideal}) of $F$. The quotient ring $T/F^\perp$ is called the \textit{apolar algebra} of $F$.

A set of reduced points $\mathbb{X}\subset\P(S_1)=\mathbb{P}^n$ is called an \textit{apolar set of $F$} (or \textit{apolar to $F$}) if the (saturated and radical) defining ideal $I_\mathbb{X}$ of $\mathbb{X}$ is contained in $F^\perp$. A key ingredient to study Waring rank and decomposition of $F$ is the Apolarity Lemma (see e.g. \cite[section 1.1]{IK} for proof). We state the lemma here over the reals.

\begin{lemma}[Apolarity Lemma]\label{real apolarity}
Let $F$ be any form of degree $d$ in $S$. Then, there exist scalars $\lambda_1,\ldots, \lambda_r$ in $\R$ such that $F=\sum_{i=1}^{r} \lambda_i L_i^d$ with $L_i=c_{i0}X_0+\cdots +c_{in}X_n$ ($c_{ij}\in\R$) if and only if $I_\mathbb{X}\subset F^\perp$ (i.e. $\mathbb{X}$ is apolar to $F$), where $\mathbb{X}=\{[c_{10}:\cdots:c_{1n}],\cdots,[c_{r0}:\cdots:c_{rn}]\}$ is a set of reduced real points in $\P(S_1)$.
\end{lemma}

\begin{remark}\label{existence_real_decomp}
In \cite{IK}, The Apolarity Lemma is originally stated over an algebraically closed field. Let $\kk$ be any subfield of $\mathbb{C}$. We can always find the coefficients $\lambda_i$ in $\kk$, once the given $F$ is a form with coefficients in $\kk$ and we restrict ourselves to a set $\mathbb{X}$ of $\kk$-points, because a system of linear equations $A\mathbf{x}=\mathbf{b}$ in which the scalars $b_k$ and the entries $A_{ij}$ are in $\kk$ has a solution indeed in $\kk$ whenever it has a solution in $\CC$. The existence of $\lambda_i$ among the complex numbers is given by the classical Apolarity Lemma.
\end{remark}

%

For any graded $\kk$-algebra $R=\bigoplus_{d\in\mathbb{Z}_{\geq 0} } R_d$, Hilbert function $\HF(R,d)$ is defined as $\dim_{\kk}R_d$. Throughout the paper, for any homogeneous ideal $I\subset T$ we use definitions for (co)dimension and degree of $I \subset T$ by the notions coming from the associated Hilbert polynomial of a $\R$-algebra $T/I$.

For later use, we state

\begin{proposition}\label{Prop_FromSatToRad}
Suppose that $J$ is a saturated ideal in $T$ and $\mathbb{X}:=V(J)$ are distinct $\R$-points in $\P(S_1)$. If $\HF(T/J,d)=|\mathbb{X}|$ for all $d\gg0$, then $J=I_{\mathbb{X}}$ (in particular, $J$ is a radical ideal in $T$).
\end{proposition}

\begin{proof}
Obviously, $J\subseteq I_{V(J)}(=I_{\mathbb{X}})$. For all $d\gg0$, $\HF(J,d)=\HF(I_{\mathbb{X}},d)=\dim~T_d-|\mathbb{X}|$, so there exists $d_0$ such that $(J)_d=(I_{\mathbb{X}})_d$ for all $d\ge d_0$. Let $\mathbf{m}=(x_0,x_1,\ldots,x_n)$, the irrelevant maximal ideal in $T$. For any $g\in (I_{\mathbb{X}})_d$ for $d<d_0$, $\mathbf{m}^k\cdot g\in (I_{\mathbb{X}})_{\ge d+k}=(J)_{\ge d+k}$ for a sufficiently large $k$. Since $J$ is saturated, $g\in(J)_d$, which implies $(I_{\mathbb{X}})_d\subseteq(J)_d$. Thus, we have $J=I_{\mathbb{X}}$.
\end{proof}

\begin{remark}\label{sat_not_enough_overR}
Over $\R$, under the assumption that $J$ is saturated and $V(J)$ consists of distinct points in $\P_{\R}^n$, we need the condition `$\HF(T/J,d)=|\mathbb{X}|$ for all $d\gg0$' in Proposition \ref{Prop_FromSatToRad} to guarantee that $J$ is the (vanishing) ideal of the real points which $J$ defines. Indeed, consider an ideal $J=\langle (x-2y)(2x+y)(x^2+y^2) \rangle$ in $T=\R[x,y]$. $J$ is saturated (and radical) in $T$ and defines $\mathbb{X}=\{[2:1],[1:-2]\}$ over $\R$, but $J\neq I_{\mathbb{X}}=\langle (x-2y)(2x+y) \rangle$.

Note also that Hilbert function of any $T$-module $M$ does not change under the base extension $\R\hookrightarrow \CC$. In other words, it holds that $\HF(M,d)=\HF(M\otimes_\R \CC,d)$  for every $d$ (see e.g. \cite[corollary 5.1.20]{KR1}). Being saturated (or radical) also does not change under the extension.
\end{remark}

We also recall the notion of Stanley decomposition of an $\R$-algebra $T/I$ and some basic results on it (see e.g. \cite{SW} for more details).

\begin{definition}\label{Stanley decomposition}
Let $I$ be an ideal in $T$ and $R=T/I$. A \textit{Stanley decomposition of $R$} is a representation of $R$ as the (internal) direct sum of $\R$-vector spaces $R=\bigoplus_{\alpha\in A} \mathbf{x}^\alpha~\mathbb{R}[\sigma_\alpha]$, where $A$ is a finite subset of $\mathbb{Z}_{\geq 0}^{n+1}$, $\mathbf{x}^\alpha$ denotes, by abuse of notation, the image in $R$ of monomial $x_0^{\alpha_0}x_1^{\alpha_1}\cdots x_n^{\alpha_n}$ for $\alpha=(\alpha_0,\alpha_1,\ldots,\alpha_n)$, and $\sigma_\alpha$ is a (possibly empty) subset of the variables $\{x_0,x_1,\ldots,x_n\}$. 
\end{definition}


\begin{proposition}\label{Stanley}
For any homogeneous ideal $I\subset T$, $T/I$ admits a Stanley decomposition. Further, for a given Stanley decomposition $T/I=\bigoplus_{\alpha\in A} \mathbf{x}^\alpha~\mathbb{R}[\sigma_\alpha]$, Hilbert function of $T/I$ is given by 
\begin{equation}\label{hilb_for}
\HF(T/I,t)=\sum_{\alpha\in A} \HF(\mathbb{R}[\sigma_{\alpha}],t-|\alpha|)\quad\textrm{where $|\alpha|=\sum_i \alpha_i$.}
\end{equation}
Finally, for any $t>\max_{\alpha\in A}\{|\alpha|\}$, $\HF(T/I,t)$ agrees with a polynomial function, which is Hilbert polynomial of $T/I$.
\end{proposition}


In the proof of Theorem \ref{main_thm}, we need to calculate the initial ideal. For this aim, we recall a lemma about \textit{subresultants} and \textit{polynomial remainder sequences} (basic ideas on this topic can be found in \cite{C}). The following lemma is a simple version of it for our purposes and we provide a proof for completeness.
 
\begin{lemma}\label{prs}
Let $f=a_nx^n+a_{n-1}x^{n-1}+\cdots+a_0$ and $g=b_mx^m+b_{m-1}x^{m-1}+\cdots+b_0$ be two real univariate polynomials with degrees $n, m$ ($n\geq m$). Then, for any $1\leq i\leq m$, there exist polynomials $u_i,v_i$ with degrees at most $m-i,n-i$, respectively, such that $u_if+v_ig$ has degree less than or equal to $i-1$. Moreover, each coefficient of $u_i$, $v_i$ and of $u_i f+v_i g$ can be expressed by the determinant of a submatrix of the Sylvester matrix of $f$ and $g$.
\end{lemma}
\begin{proof}
For $1\le i\le m$, consider the map $\phi_i : \mathbb{R}[x]_{\le m-i}\times \mathbb{R}[x]_{\le n-i}\rightarrow \mathbb{R}[x]_{\le n+m-i}$, which is given by $\phi_i(A,B)=Af+Bg$. Actually, this map can be understood by the linear transformation $$\tilde{\phi}_i : \mathbb{R}^{m-i+1}\times\mathbb{R}^{n-i+1}\simeq\mathbb{R}^{n+m-2i+2}\rightarrow \mathbb{R}^{n+m-i+1}$$ with a $(n+m-i+1)\times(n+m-2i+2)$-matrix presentation
$$M=\small{\begin{pmatrix} 
a_n 		& 		&		&b_m 	& 		&	\\	
a_{n-1} 	&\ddots 	&		&b_{m-1} 	&\ddots 	&	\\	
\vdots 	&\ddots 	&a_n 	&\vdots 	&\ddots 	&b_m \\
a_0 		& \ddots 	& a_{n-1} 	& b_0 	& \ddots 	& b_{m-1} \\	
		& \ddots 	& \vdots 	& 		& \ddots 	& \vdots \\
		& 		& a_0   	&		& 		& b_0   
\end{pmatrix}}$$
which is a submatrix of the Sylvester matrix of $f$ and $g$. Now, we want to find a solution $\mathbf{x}\in\mathbb{R}^{n+m-2i+2}$ to the equation
$$M\mathbf{x}=\tiny{\begin{pmatrix} 0\\ \vdots \\ 0 \\ \ast \\ \vdots \\ \ast\end{pmatrix}}$$
with the top $n+m-2i+1$ entries being zeros. For the part of top $n+m-2i+1$ rows, the equation becomes $$\widetilde{M}\mathbf{x}=\mathbf{0}$$ where $\widetilde{M}$ is a $(n+m-2i+1)\times (n+m-2i+2)$ submatrix of $M$. Since the difference between the number of rows and the number of columns in $\widetilde{M}$ is $1$, a solution $\mathbf{x}$ can be given by determinants of maximal submatrices of $\widetilde{M}$. In other words, for $0\leq j\leq n+m-2i+1$, if $\widetilde{M}^j$ is the submatrix of $\widetilde{M}$ by removing $(j+1)$-th column, then $$\mathbf{x}=\ds\big[(-1)^j\det(\widetilde{M}^j)\big]_{j=0}^{n+m-2i+1}~~.$$
Accordingly, the coefficients of polynomials $u_i$ and $v_i$ are given by the vector $\mathbf{x}$ as follows:
\begin{equation}\label{coeff_u,v}
u_i=\sum_{j=0}^{m-i}(-1)^j\det(\widetilde{M}^j) x^{m-i-j}~, \quad
v_i=\sum_{j=0}^{n-i} (-1)^{m-i+1+j}\det(\widetilde{M}^{m-i+1+j})x^{n-i-j}~.
\end{equation}
Further, the bottom $\ast$-part of the equation $M\mathbf{x}=\tiny{\begin{pmatrix} \mathbf{0} \\ \ast \\ \vdots \\ \ast\end{pmatrix}}$ can be decided up to sign by determinants of $\widetilde{M}_k$'s which are $(n+m-2i+2)\times (n+m-2i+2)$ submatrices of $M$ and the submatrix $\widetilde{M}_k$ is obtained by adding the $(n+m-2i+2+k)$-th row of $M$ to $\widetilde{M}$ for $0\le k\le i-1$. Then, the resulting polynomial $u_if+v_ig$, which has degree at most $(i-1)$, can be written in the form
\begin{equation}\label{uf+vg}
(-1)^{n+m-2i+1}\cdot\sum_{k=0}^{i-1} \det(\widetilde{M}_k)x^{i-1-k}\qquad.
\end{equation}
\end{proof}

\section{Proof of the main result}\label{sect_upp_bd}
\subsection{The ideal $J_\mathbf{a}$}\label{subsect_3-1}

Recall that $S=\R[X_0,\ldots,X_n]$ is a polynomial ring over the reals and $T=\R[x_0,\ldots,x_n]$ is the dual ring of $S$. For a sequence of positive integers $\mathbf{a}=(a_0,a_1,\ldots,a_n)$ with each $a_i>0$, consider a monomial $\mathbf{X}^{\mathbf{a}}=X_0^{a_0}X_1^{a_1}\cdots X_n^{a_n}$ in $S$. Then the apolar ideal $(\mathbf{X}^\mathbf{a})^\perp$ in $T$ is given by $(x_0^{a_0+1},x_1^{a_1+1},\ldots,x_n^{a_n+1})$. We would like to find an appropriate ideal of $\R$-points $J\subset(\mathbf{X}^\mathbf{a})^\perp$ which computes our upper bound in Theorem \ref{main_thm}. To do this, we first need to introduce the following families of polynomials with a real parameter.

\begin{definition}\label{Fij}Let $t$ be any nonzero real number. For a given sequence $\mathbf{a}=(a_0,a_1,\ldots,a_n)$ of odd numbers, define the polynomials of $T$, $F_{i,j}^{\mathbf{a}}(t)$ by
$$F_{i,j}^{\mathbf{a}}(t):=\prod_{k=-\lfloor\frac{a_i}{4}\rfloor-\lfloor\frac{a_j+2}{4}\rfloor}^{\lfloor\frac{a_i+2}{4}\rfloor+\lfloor\frac{a_j}{4}\rfloor}(x_i-t^kx_j)(x_i+t^kx_j)\quad\textrm{for any $i,j$ with $0\leq i<j\leq n$}.$$
\end{definition}
\begin{definition}\label{Gij} Consider the function $\epsilon:\mathbb{Z}\rightarrow \{0,1\}$ with
$$\epsilon(p):=\begin{cases}
0 & p~\text{is odd}\\
1 & p~\text{is even}
\end{cases}.
$$ 
For any sequence $\mathbf{a}$ of positive integers with length $n+1$, let $\mathbf{a}'$ be the associated sequence of odd numbers given by 
\begin{equation}\label{a_prime}
a'_i=a_i-\epsilon(a_i)~,\quad i=0,\ldots,n~.
\end{equation}
For $i,j$ with $0\leq i<j\leq n$, define $G_{i,j}^{\mathbf{a}}(t)$ as
$$G_{i,j}^{\mathbf{a}}(t):=x_i^{\epsilon(a_i)}x_j^{\epsilon(a_j)}F_{i,j}^{\mathbf{a}'}(t)~.$$
In other words, 
$$G_{i,j}^{\mathbf{a}}(t)=
\begin{cases}
F_{i,j}^{\mathbf{a}'}(t) & \mathbf{a}_i~\text{is odd}, ~\mathbf{a}_j~\text{is odd}\\
x_iF_{i,j}^{\mathbf{a}'}(t) &\mathbf{a}_i~\text{is even}, ~\mathbf{a}_j~\text{is odd}\\
x_jF_{i,j}^{\mathbf{a}'}(t) & \mathbf{a}_i~\text{is odd}, ~\mathbf{a}_j~\text{is even}\\
x_ix_jF_{i,j}^{\mathbf{a}'}(t) & \mathbf{a}_i~\text{is even},~ \mathbf{a}_j~\text{is even}~.
\end{cases}
$$
From now on, instead of $F_{i,j}^{\mathbf{a}}(t)$ and $G_{i,j}^{\mathbf{a}}(t)$ we often use the notations of $F_{i,j}^{\mathbf{a}}$ and $G_{i,j}^{\mathbf{a}}$ without $t$ (or even simpler $F$ and $G$) where everything is clear.
\end{definition}

\begin{remark}\label{F rmk} 
We have some remarks on the forms $F_{i,j}^{\mathbf{a}}(t)$ and $G_{i,j}^\mathbf{a}(t)$.
\begin{enumerate}
\item (Hyperbolicity of $F$ and $G$) First of all, we would like to note that for any $t\in\R$ the binary forms $F_{i,j}^{\mathbf{a}}(t)$ and $G_{i,j}^\mathbf{a}(t)$ are real hyperbolic polynomials in that all their roots are real (see e.g. \cite[section 2]{BS} for more on this). In particular, if $t\neq \pm1$, the forms $F_{i,j}^{\mathbf{a}}(t)$ and $G_{i,j}^\mathbf{a}(t)$ have distinct real roots.
\item For any sequence of odd numbers $\mathbf{a}$, $F_{i,j}^{\mathbf{a}}(t)$ has degree 
$$2\cdot\left(\lfloor\frac{a_i+2}{4}\rfloor+\lfloor\frac{a_j}{4}\rfloor+\lfloor\frac{a_i}{4}\rfloor+\lfloor\frac{a_j+2}{4}\rfloor+1\right)=a_i+a_j~.$$
Similarly, we have $\deg(G_{i,j}^\mathbf{a}(t))=\epsilon(a_i)+\epsilon(a_j)+\deg(F_{i,j}^{\mathbf{a}'})=\epsilon(a_i)+\epsilon(a_j)+a_i'+a_j'=a_i+a_j$ for any sequence of positive integers $\mathbf{a}$. Furthermore, from the definition, it also holds that $F_{i,j}^{\mathbf{a}}(t)=F_{i,j}^{\mathbf{a}}(-t)$ and $G_{i,j}^\mathbf{a}(t)=G_{i,j}^\mathbf{a}(-t)$ for all $t\in\R$ (i.e. the $F_{i,j}^\mathbf{a}$ and $G_{i,j}^\mathbf{a}$ are even functions).
\item (Expansions of $F$ and $G$) For a sequence of odd numbers $\mathbf{a}=(a_0,a_1,\ldots,a_n)$, $F_{i,j}^{\mathbf{a}}(t)$ can be expanded as
\begin{equation}\label{F_expand}
F_{i,j}^\mathbf{a}(t)=\sum_{d=0}^{\frac{a_i+a_j}{2}}C_{i,j,d}^{\mathbf{a}}(t)x_i^{a_i+a_j-2d}x_j^{2d}
\end{equation}
where each coefficient is given by
\begin{equation}\label{F_coeff}
C_{i,j,d}^{\mathbf{a}}(t)=(-1)^{d}\sum_{\substack{\mathcal{I}\subset \mathcal{J} \\ |\mathcal{I}|=d}}\prod_{k\in \mathcal{I}}t^{2k}
\end{equation}
for $\mathcal{J}=\{-\lfloor\frac{a_i}{4}\rfloor-\lfloor\frac{a_j+2}{4}\rfloor,\ldots,\lfloor\frac{a_i+2}{4}\rfloor+\lfloor\frac{a_j}{4}\rfloor\}$.
Note that $C_{i,j,0}^{\mathbf{a}}(t)=1$ and in general each coeffient $C_{i,j,d}^{\mathbf{a}}(t)$ can be viewed formally as a Laurent polynomial in $\mathbb{R}[t,t^{-1}]$. Its top degree term has degree equal to twice the sum of $d$ greatest elements in $\mathcal{J}$. Let us denote this top degree by $Q_{i,j}^\mathbf{a}(d)$, which is given by the quadratic polynomial of $d$ as the following form
\begin{align}\label{Q_ij}
Q_{i,j}^\mathbf{a}(d):&=\mathrm{top~degree}(C_{i,j,d}^{\mathbf{a}}(t))=\sum_{s=0}^{d-1} \left(2(\lfloor\frac{a_i+2}{4}\rfloor+\lfloor\frac{a_j}{4}\rfloor)-2s\right)\\ \nonumber
&=2d\left(\lfloor\frac{a_i+2}{4}\rfloor+\lfloor\frac{a_j}{4}\rfloor\right)-(d-1)d\\ \nonumber
&=-d^2+d\left(1+2\lfloor\frac{a_i+2}{4}\rfloor+2\lfloor\frac{a_j}{4}\rfloor\right).
\end{align}
For a sequence of positive integers $\mathbf{a}$, let $\mathbf{a}'$ be the associated sequence of odd numbers as in Definition \ref{Gij}. Then $G_{i,j}^{\mathbf{a}}(t)$ can be also rewritten as the following form
\begin{equation}\label{G_expand}
G_{i,j}^\mathbf{a}(t)=\sum_{d=0}^{\frac{a_i'+a_j'}{2}}C_{i,j,d}^{\mathbf{a}'}(t)x_i^{a_i+a_j'-2d}x_j^{2d+\epsilon(a_j)}~.
\end{equation}
\item ($F_{i,j}^{\mathbf{a}}(t)$ and $G_{i,j}^{\mathbf{a}}(t)$ are in $(\mathbf{X}^\mathbf{a})^\perp$ for any nonzero real $t$) Using the above expansions, we observe that for any sequence of odd numbers $\mathbf{a}$, $F_{i,j}^{\mathbf{a}}(t)\in (\mathbf{X}^\mathbf{a})^\perp$ because for any $t\neq0\in\R$
\begin{multline*}
F_{i,j}^{\mathbf{a}}(t)=x_i^{a_i+a_j}+\cdots+C_{i,j,\frac{a_j-1}{2}}^{\mathbf{a}}(t) x_i^{a_i+1}x_j^{a_j-1}+C_{i,j,\frac{a_j+1}{2}}^{\mathbf{a}}(t) x_i^{a_i-1}x_j^{a_j+1}+\cdots\\
\in (x_i^{a_i+1},x_j^{a_j+1}).
\end{multline*}

Similarly, we can see that $G_{i,j}^{\mathbf{a}}(t)\in (\mathbf{X}^\mathbf{a})^\perp$ for any sequence of positive integers $\mathbf{a}$, since
\begin{multline*}
G_{i,j}^{\mathbf{a}}(t)=x_i^{a_i+a_j'}x_j^{\epsilon(a_j)}+\cdots+C_{i,j,\frac{a_j'-1}{2}}^{\mathbf{a}'}(t) x_i^{a_i+1}x_j^{a_j-1}+C_{i,j,\frac{a_j'+1}{2}}^{\mathbf{a}'}(t) x_i^{a_i-1}x_j^{a_j+1}+\cdots\\
\in (x_i^{a_i+1},x_j^{a_j+1})~.
\end{multline*}
\end{enumerate}
\end{remark}

\begin{definition}\label{def_Ja} Now, we define an ideal 
\begin{equation*}
J_{\mathbf{a}}(t):=\left(\{G_{i,j}^{\mathbf{a}}(t)~:~0\leq i<j\leq n\}\right)
\end{equation*}
Note that $J_{\mathbf{a}}(t)$ becomes a homogeneous ideal in $T$ by evaluation at any nonzero parameter $t\in\R$. 
\end{definition}

\begin{remark}
An example which shows a motive for the choice in the definition of $F$, $G$ and $J_\mathbf{a}$ is the ideal $(x_0^3x_1-x_0x_1^3,x_0^3x_2-x_0x_2^3,x_1^3x_2-x_1x_2^3)$ which is contained in the apolar ideal $(X_0^2X_1^2X_2^2)^\perp=(x_0^3,x_1^3,x_2^3)$ and defines $13$ points on the real projective plane. This gives $\rank_{\mathbb{R}}(X_0^{2}X_1^{2}X_2^{2})\le13$. In fact, there are other similar ways to take a generating set of binary forms inside the apolar ideal which cuts distinct real points on the plane.

But if we do not choose the symmetry of roots of generators carefully, the saturatedness of that ideal is easily fails. For instance, the ideal $(x_0^3x_1-x_0x_1^3,4x_0^3 x_2-x_0x_2^3,x_1^3x_2-9x_1x_2^3)$ defines $9$ points and is also contained in the apolar ideal \\$(x_0^3,x_1^3,x_2^3)$. But this ideal is not saturated and $x_0x_1x_2$ is added under the saturation. Then the resulting ideal does not belong to the apolar ideal, hence we can no longer use the Apolarity Lemma. 
\end{remark}


By abuse of notation, we often use $J_{\mathbf{a}}$ to denote $J_{\mathbf{a}}(t)$ whenever the parameter $t$ is not necessary for arguments.

\begin{proposition}\label{Ja}
For a monomial $\mathbf{X}^\mathbf{a}=X_0^{a_0}X_1^{a_1}\cdots X_n^{a_n}$ with each $a_i>0$, the ideal $J_{\mathbf{a}}(t)$ is contained in $(\mathbf{X}^\mathbf{a})^\perp=(x_0^{a_0+1},x_1^{a_1+1},\ldots,x_n^{a_n+1})$ for any nonzero $t\in\R$. The zero set $V(J_{\mathbf{a}}(t))$ consists of $\frac{1}{2}(\prod_{i=0}^n(a_i+1)-\prod_{i=0}^n(a_i-1))$ real points for any nonzero real $t\neq \pm1$.
\end{proposition}
\begin{proof}
It is obvious that $J_{\mathbf{a}}(t)\subset(\mathbf{X}^\mathbf{a})^\perp$ by Remark \ref{F rmk} (4). For the common zero set, first of all, let us treat the case where all $a_i$ are odd. In this case, $G_{i,j}^{\mathbf{a}}(t)=F_{i,j}^{\mathbf{a}}(t)$ and it is easy to see that $V(J_{\mathbf{a}}(t))$ has no zero on the coordinate planes. Let us denote the points of $V(J_{\mathbf{a}}(t))$ in the positive orthant by $V(J_\mathbf{a})^+$. All other points of $V(J_{\mathbf{a}}(t))$ can be obtained by all reflections of $V(J_\mathbf{a})^+$ with respect to every axis. Since $G_{i,j}^\mathbf{a}(t)=G_{i,j}^\mathbf{a}(-t)$, from now on, assuming $t>0$.

We claim that $V(J_\mathbf{a})^+$ is the set $\Gamma=\{[t^{b_0}:t^{b_1}:\ldots:t^{b_n}]~|~-\lfloor\frac{a_i}{4}\rfloor\leq b_i\leq \lfloor\frac{a_i+2}{4}\rfloor,b_i\in \mathbb{Z}\}$. To see that $\Gamma\subset V(J_{\mathbf{a}})^+$, we need to show that every generator $G_{i,j}$ of $J$ vanishes at the point $[t^{b_0}:t^{b_1}:\ldots:t^{b_n}]$. Since $-\lfloor\frac{a_i}{4}\rfloor\leq b_i\leq \lfloor\frac{a_i+2}{4}\rfloor$ and $-\lfloor\frac{a_j}{4}\rfloor\leq b_j\leq \lfloor\frac{a_j+2}{4}\rfloor$, we get $-\lfloor\frac{a_i}{4}\rfloor-\lfloor\frac{a_j+2}{4}\rfloor\leq b_i-b_j\leq \lfloor\frac{a_i+2}{4}\rfloor+\lfloor\frac{a_j}{4}\rfloor$. Hence, $G^\mathbf{a}_{i,j}(t)([t^{b_0}:t^{b_1}:\ldots:t^{b_n}])=0$ for all $0\leq i<j\leq n$.

To prove the opposite direction, $V(J_{\mathbf{a}})^+\subset \Gamma$, let $p$ be any point $[\alpha_0:\alpha_1:\ldots:\alpha_n]$ of $V(J_{\mathbf{a}})^+$. First we see that $p$ can be written in the form $[t^{c_0}:t^{c_1}:\ldots:t^{c_n}]$ for some integers $c_i$. For any $i$ with $0\le i< n$, $\alpha_i=t^{d_i}\alpha_n$ for some $d_i\in\ZZ$, because $G^\mathbf{a}_{i,n}(t)(p)=0$ and $\alpha_i, \alpha_n>0$. So, $p=[t^{d_0}\alpha_n:t^{d_1}\alpha_n:\ldots:\alpha_n]$ and if we set $\alpha_n=t^{d_n}$ for some $d_n\in\ZZ$ (by scaling $\alpha_i$ if necessary), then $p$ is of the form $[t^{c_0}:t^{c_1}:\ldots:t^{c_n}]$ and $-\lfloor\frac{a_i}{4}\rfloor-\lfloor\frac{a_j+2}{4}\rfloor\leq c_i-c_j\leq \lfloor\frac{a_i+2}{4}\rfloor+\lfloor\frac{a_j}{4}\rfloor$ from the definition of $G^\mathbf{a}_{i,j}$. 
  
  Let $k_i=c_i+\lfloor\frac{a_i}{4}\rfloor$ and $k_l=\min\{k_0,\ldots,k_n\}$. Then $0\leq k_i-k_l=c_i+\lfloor\frac{a_i}{4}\rfloor-k_l$. So $-\lfloor\frac{a_i}{4}\rfloor\leq c_i-k_l=(c_i-c_l)-\lfloor\frac{a_l}{4}\rfloor\leq (\lfloor\frac{a_i+2}{4}\rfloor+\lfloor\frac{a_l}{4}\rfloor)-\lfloor\frac{a_l}{4}\rfloor=\lfloor\frac{a_i+2}{4}\rfloor$ for all $i=0,\ldots,n$. Thus, we can conclude that $p=t^{-k_l}\cdot[t^{c_0}:t^{c_1}:\ldots:t^{c_n}]\in \Gamma$. 

Next, for any real $t\neq\pm1$, let us count the number of points of $\Gamma$. Since $[t^{b_0}:t^{b_1}:\ldots:t^{b_n}]=[t^{b_0+1}:t^{b_1+1}:\ldots:t^{b_n+1}]$, we need to consider the equivalence among points in the description of $\Gamma$. Note that a representative in each equivalence can be chosen as $[t^{b_0}:t^{b_1}:\ldots:t^{b_n}]$ with at least one index $i$ such that $-\lfloor\frac{a_i}{4}\rfloor=b_i$. Hence the number of elements in $\Gamma$ is given by $$\prod_{i=0}^n(\lfloor\frac{a_i}{4}\rfloor+\lfloor\frac{a_i+2}{4}\rfloor+1)-\prod_{i=0}^n(\lfloor\frac{a_i}{4}\rfloor+\lfloor\frac{a_i+2}{4}\rfloor)~.$$
Since each $a_i$ is odd and $\lfloor\frac{a_i}{4}\rfloor+\lfloor\frac{a_i+2}{4}\rfloor+1=\frac{a_i+1}{2}$, $\lfloor\frac{a_i}{4}\rfloor+\lfloor\frac{a_i+2}{4}\rfloor=\frac{a_i-1}{2}$ regardless of $a_i=4k+1$ or $4k+3$, we obtain
$$|V(J_{\mathbf{a}})^+|=|\Gamma|=\prod_{i=0}^n \frac{a_i+1}{2}-\prod_{i=0}^n \frac{a_i-1}{2}~.$$
Moreover, by considering all the $n$ coordinate reflections, we find
$$|V(J_{\mathbf{a}})|=2^n\cdot(\prod_{i=0}^n\frac{a_i+1}{2}-\prod_{i=0}^n \frac{a_i-1}{2})=\frac{1}{2}\left(\prod_{i=0}^n(a_i+1)-\prod_{i=0}^n(a_i-1)\right)~.$$

Now, let us prove the statement for any sequence $\mathbf{a}=(a_0,a_1,\cdots,a_n)$. Let $E$ be the set of indices $i$ for which $a_i$ is even. We will use a double induction on $n$, the length of $\mathbf{a}$, and on $|E|$, the size of $E$. One initial condition $|E|=0$ (i.e. the case of all $a_i$ being odd) is checked above. Another initial case $n=1$ follows as $G^\mathbf{a}_{0,1}(t)$ has $a_0+a_1$ distinct real roots when $t\neq \pm 1$.

Consider any sequence $\mathbf{a}=(a_0,a_1,\ldots,a_n)$ with $|E|=k$ and assume that the statement does hold if the number of even indices is less than $k$ for any sequence of the same length or if the length of a sequence is shorter than $\mathbf{a}$. Take an index $l\in E$, consider $\mathbf{a}'=\mathbf{a}-\mathbf{e}_l$, where $\mathbf{e}_l$ is the $l$-th coordinate vector. Since $a_l$ is even, all polynomials $G_{i,l}^{\mathbf{a}}$ and $G_{l,j}^{\mathbf{a}}$ have the factor $x_l$. This property implies that $V(J_\mathbf{a})$ is a union of  two disjoint subsets - one subset from $x_l=0$ and the other $x_l\neq0$. If $x_l$ is nonzero, computing zeros of this subset is reduced to the case $\mathbf{a}'$. On the other hand, if $x_l$ is zero, the common zeros are same as the case $\mathbf{a}'':=(a_0,a_1,\ldots,a_{l-1},a_{l+1},\ldots,a_n)$. By the induction hypothesis, we have 
\begin{align*}
|V(J_\mathbf{a})|&=|V(J_{\mathbf{a}'})|+|V(J_{\mathbf{a}''})|\\
&=\frac{1}{2}\bigg((a_l-1+1)\prod_{i\neq l}^n (a_i+1)-(a_l-1-1)\prod_{i\neq l}^n (a_i-1)\bigg)+\frac{1}{2}\bigg(\prod_{i\neq l}^n (a_i+1)-\prod_{i\neq l}^n (a_i-1)\bigg)\\
&=\frac{1}{2}\left(\prod_{i=0}^n (a_i+1)-\prod_{i=0}^n (a_i-1)\right)~.
\end{align*}
\end{proof}

\begin{example}\label{Ex543, Ex444}
As an illustration of our method, we present an apolar set for $X_0^5X_1^4X_2^3$ and $X_0^4X_1^4X_2^4$. First, let $\mathbf{a}=\{5,4,3\}$. Then $J_\mathbf{a}(t)=(G_{0,1}^\mathbf{a},G_{0,2}^\mathbf{a},G_{1,2}^\mathbf{a})$ and
\begin{align*}
&G_{0,1}^\mathbf{a}=x_1(x_0^2-\frac{1}{t^4}x_1^2)(x_0^2-\frac{1}{t^2}x_1^2)(x_0^2-x_1^2)(x_0^2-t^2x_1^2)~\in(x_0^6,x_1^5),\\
&G_{0,2}^\mathbf{a}=(x_0^2-\frac{1}{t^4}x_2^2)(x_0^2-\frac{1}{t^2}x_2^2)(x_0^2-x_2^2)(x_0^2-t^2x_2^2)~\in(x_0^6,x_2^4),\\
&G_{1,2}^\mathbf{a}=x_1(x_1^2-\frac{1}{t^2}x_2^2)(x_1^2-x_2^2)(x_1^2-t^2x_2^2)~\in(x_1^5,x_2^4)~,
\end{align*}
where we can check that $(G_{0,1}^\mathbf{a},G_{0,2}^\mathbf{a},G_{1,2}^\mathbf{a})\subset (x_0^6,x_1^5,x_2^4)$, the apolar ideal. Since $E$, the set in the proof of Proposition \ref{Ja}, is just $\{1\}$, $V(J_\mathbf{a})$ consists of a union of two disjoint parts, i.e. $\{x_1=0\}$ and $\{x_1\neq 0\}$. Note that the zeros in $\{x_1\neq 0\}$ are given by the points in $V(J_{\mathbf{a}'})$ where $\mathbf{a}'=\{5,3,3\}$. The points in $V(J_{\mathbf{a}'})$ can be represented by $[t^{k_0};t^{k_1};t^{k_2}]$ with $-\lfloor\frac{a_i}{4}\rfloor\leq k_i\leq \lfloor\frac{a_i+2}{4}\rfloor$ and at least one $k_i$ is equal to the lower bound.
To simplify the situation, let $k_i'=k_i+\lfloor\frac{a_i}{4}\rfloor$ and replace all $x_i$ by $t^{-\lfloor\frac{a_i}{4}\rfloor}x_i$. Then $0\leq k_i'\leq  \lfloor\frac{a_i+2}{4}\rfloor+\lfloor\frac{a_i}{4}\rfloor=\lfloor\frac{a_i-1}{2}\rfloor$ and at least one of $k_i'$ is zero.
In this case, $0\leq k_0'\leq 2, 0\leq k_1'\leq 1, 0\leq k_2'\leq 1$ and at least one of $k_i'$ is zero. Among the $3\cdot 2\cdot 2=12$ possible choices, the two caes $(k_0',k_1',k_2')=(1,1,1)$ and $(2,1,1)$ do not satisfy the condition that at least one of $k_i'$ is zero.

Hence, considering axis-reflections on $x_1, x_2$, there are $10\cdot 2^2=40$ nonzero points 
\begin{align*}
\{[1;\pm 1;\pm 1],[1;\pm 1;\pm t],[1;\pm t;\pm 1],[1;\pm t;\pm t],[t;\pm 1;\pm 1],\\
[t;\pm 1;\pm t],[t;\pm t;\pm 1],[t^2;\pm 1;\pm 1],[t^2;\pm 1;\pm t],[t^2;\pm t;\pm 1]\}
\end{align*}
for any real $t\neq 0,\pm1$. In $V(J_\mathbf{a})$, there are  another points in $\{x_1=0\}$. These points have the form $[t^{k_0'},0,t^{k_2'}]$ where $0\leq k_0'\leq 2$ and $0\leq k_2'\leq 1$. Among $3\cdot 2=6$ possible choices, $(k_0',k_2')=(1,1)$ and $(2,1)$ do not satisfy the condition that at least one of $k_i'$ is zero. Hence, considering axis-reflections there are $4\cdot 2=8$ points such as
 $$\{[1;0;\pm 1],[1;0;\pm t],[t;0;\pm 1],[t^2;0;\pm 1]\}.$$ 
They give a total of $48$ points and we have $\frac{1}{2}(4\cdot 5\cdot 6-2\cdot 3\cdot 4)=48$.
\smallskip

Now, let us take $\mathbf{a}=\{4,4,4\}$. Then $J_{\mathbf{a}}=(G_{0,1}^\mathbf{a},G_{0,2}^\mathbf{a},G_{1,2}^\mathbf{a})$
\begin{align*}
&G_{0,1}^\mathbf{a}=x_0x_1(x_0^2-\frac{1}{t^2}x_1^2)(x_0^2-x_1^2)(x_0^2-t^2x_1^2)~\in(x_0^5,x_1^5),\\
&G_{0,2}^\mathbf{a}=x_0x_2(x_0^2-\frac{1}{t^2}x_2^2)(x_0^2-x_2^2)(x_0^2-t^2x_2^2)~\in(x_0^5,x_2^5),\\
&G_{1,2}^\mathbf{a}=x_1x_2(x_1^2-\frac{1}{t^2}x_2^2)(x_1^2-x_2^2)(x_1^2-t^2x_2^2)~\in(x_1^5,x_2^5)~,
\end{align*}
In this case, $E=\{0,1,2\}$ and $0\leq k_0'\leq 1, 0\leq k_1'\leq 1, 0\leq k_2'\leq 1$. There are $7$ disjoint subsets of $V(J_\mathbf{a})$ which are 
\begin{align*}
x_0\neq0,x_1\neq 0, x_2\neq 0,&\quad x_0\neq 0, x_1\neq 0, x_2=0, \quad x_0\neq 0, x_1= 0, x_2\neq0,\\
x_0= 0, x_1\neq 0, x_2\neq0,&\quad x_0\neq 0, x_1= 0, x_2=0, \quad  x_0=0, x_1\neq 0, x_2=0,\\
 x_0= 0, x_1= 0, x_2\neq0~. &&
 \end{align*}
Accordingly, the points from these 7 subsets are given by
\begin{align*}
&\{[1;\pm 1;\pm 1],[1;\pm 1;\pm t],[1;\pm t;\pm 1],[t;\pm 1;\pm 1],[1;\pm t;\pm t],[t;\pm 1;\pm t],[t;\pm t;\pm 1]\}~,\\
&\{[1;\pm 1;0],[1;\pm t;0],[t;\pm 1;0]\}~,\{[1;0;\pm 1],[1;0;\pm t],[t;0;\pm 1]\},~\{[0;1;\pm 1],[0;1;\pm t],[0;t;\pm 1]\}~,\\
&\{[1;0;0]\}~,\{[0;1;0]\}~,\{[0;0;1]\}
\end{align*}
for any nonzero real $t\neq \pm1$. Hence we have $7\cdot 2^2+3\cdot 2^1+3\cdot 2^1+3\cdot 2^1+1\cdot 2^0+1\cdot 2^0+1\cdot 2^0=49$ points.
This is same as the number $\frac{1}{2}(5\cdot 5\cdot 5-3\cdot 3\cdot 3)=49$. 

Further, we can check that points of $V(J_{\{4,4,4\}})$ with $x_0\neq 0$ and $x_2\neq 0$ is included in $V(J_{\{5,4,3\}})$. 
This fact is also true in general : If $a_k\geq a_l$ are nonzero even numbers, then the points in $V(J_{\mathbf{a}})$ with $x_k\neq 0$ and $x_l\neq 0$ is included in the set $V(J_\mathbf{a}')$ with $x_k\neq 0$ and $x_l\neq 0$, where $\mathbf{a}'=\{a_0,\ldots,a_k+1,\ldots,a_l-1,\ldots,a_n\}$. 
\end{example}

\subsection{Initial ideal of $J_\mathbf{a}$ with respect to the graded reverse lexicographic order}\label{subsect_3-2}

In what follows, we consider the initial ideal of $J_\mathbf{a}$ with respect to the \textit{graded reverse lexicographic order}. With the help of information on the structure of $in (J_\mathbf{a}(t))$, in subsection \ref{subsect_3-3} we show that at least for all but finitely many real $t$, $J_\mathbf{a}(t)$ defines some apolar set \textit{ideal-theorectically}, by which one can compute the upper bound in Theorem \ref{main_thm}. 

We proceed in the following steps : 
\begin{enumerate}
\item Begin with the case that the last index $a_n$ of the exponent sequence $\mathbf{a}$ is an odd number.
\item Define a monomial ideal $M_\mathbf{a}$ which is a candidate for the initial ideal of $J_\mathbf{a}$. Check that the degree of this ideal $M_\mathbf{a}$ is equal to the cardinality of $V(J_\mathbf{a})$ (Lemma \ref{degree of monomial}). 
\item Prove that $M_\mathbf{a}\subset in(J_\mathbf{a})$ by finding a family of polynomials in $J_\mathbf{a}$ whose leading terms are in $M_\mathbf{a}$ (Proposition \ref{initial supset}).
\item Using Lemma \ref{prop_deg_drop}, a structural property of $M_\mathbf{a}$, we settle $M_\mathbf{a}=in(J_\mathbf{a})$ for the case of $a_n$ odd in Proposition \ref{in_J_a_odd}.
\item Generalize this result to the case of any sequence $\mathbf{a}$ with $a_n$ even in Proposition \ref{in_J_a_even}.
\end{enumerate}

\begin{definition}\label{beta}
Let $\mathbf{a}$ be a sequence of positive integers $(a_0,a_1,\ldots,a_n),~a_0\geq a_1\geq \cdots\geq a_n$ with $a_n$ odd and $\mathbf{i}=(i_0,i_1,\ldots,i_{n-1})$ be a lattice point in the simplex 
\begin{align*}
\Delta^{n-1}(\lambda_\mathbf{a}):=\{(i_0,i_1,\ldots,i_{n-1})~|~i_0+i_1+\ldots+i_{n-1}=\lambda_\mathbf{a}, ~\textrm{for all}~ i_k\geq 0\}~\\\textrm{(let $\lambda_\mathbf{a}:=\frac{a_n+1}{2}$)}.
\end{align*}
Then we define an 1-1 mapping $\beta(\mathbf{a},-):\Delta^{n-1}(\lambda_\mathbf{a})\to\mathbb{Z}_{\ge0}^n$ and $\mathbf{x}^{\beta(\mathbf{a},\mathbf{i})}$ as follows :
$$\mathbf{x}^{\beta(\mathbf{a},\mathbf{i})}:=x_0^{\beta_0(\mathbf{a},\mathbf{i})}x_1^{\beta_1(\mathbf{a},\mathbf{i})}\cdots x_{n-1}^{\beta_{n-1}(\mathbf{a},\mathbf{i})},\quad~\textrm{for all}~\mathbf{i}\in \Delta^{n-1}(\lambda_\mathbf{a})$$ where
$$\beta_k(\mathbf{a},\mathbf{i})=
\begin{cases}
0, & \text{if}\quad i_k=0\\
a_k-1-2(i_0+i_1+\cdots+i_{k-1})+2i_k, & \text{otherwise}
\end{cases}\quad\textrm{for $0\leq k\leq n-1$.}$$
In other words, 
$$\begin{bmatrix} \beta_0 \\ \beta_1\\ \beta_2 \\ \vdots \\ \beta_{n-1}\end{bmatrix}
=\delta(\mathbf{i})\cdot\left(\begin{bmatrix} a_0-1 \\ a_1-1\\ a_2-1 \\ \vdots \\  a_{n-1}-1\end{bmatrix}
+\begin{bmatrix} 
2 & 0 & 0&\cdots & 0\\
-2 & 2 & 0&\cdots & 0\\
-2 & -2 & \ddots&\ddots & 0\\
\vdots & \vdots & \ddots& \ddots & \vdots \\
-2 & -2 & \cdots &-2& 2
\end{bmatrix}
\begin{bmatrix}i_0 \\ i_1\\ i_2 \\ \vdots \\ i_{n-1}\end{bmatrix}\right)
$$
where $\delta(\mathbf{i})$ be the diagonal matrix $diag(\delta(i_0),\delta(i_1),\ldots,\delta(i_{n-1}))$, where $\delta$ is a function with $\delta(k)=0$ for $k=0$ and $\delta(k)=1$ for others.
\end{definition}

\begin{remark}\label{beta rmk}
For two lattice points $\mathbf{i}, \mathbf{j}\in\Delta^{n-1}(\lambda_\mathbf{a})$, if $i_0+i_1+\cdots+i_s=j_0+j_1+\cdots+j_s$ and $i_{s'}=j_{s'}$ for all $s'>s$, then by the definition $\mathbf{x}^{\beta(\mathbf{a},\mathbf{i})}$ and $\mathbf{x}^{\beta(\mathbf{a},\mathbf{j})}$ have same exponents for $x_{s'},\forall s'>s$. 
\end{remark}

Using this definition, we have a candidate for $in(J_\mathbf{a})$ whose monomial generators are indexed by the \textit{lattice points in the simplex} $\Delta^{n-1}(\lambda_\mathbf{a})$.

\begin{definition}\label{def_Ma}
For a sequence of positive integers $\mathbf{a}=(a_0,\ldots,a_n)$, $a_0\geq a_1\geq \ldots \geq a_n$ with $a_n$ odd, the monomial ideal $M_\mathbf{a}$ is defined by
$$M_\mathbf{a}:=\left(\{\mathbf{x}^{\beta(\mathbf{a},\mathbf{i})};\mathbf{i}\in\Delta^{n-1}(\lambda_\mathbf{a})\}\right)~.$$
Note that there is no $x_n$ in the generators of $M_\mathbf{a}$ from Definition \ref{beta}.
\end{definition}

First, we calculate the degree of this monomial ideal and check some special structural property. 
\begin{lemma}\label{degree of monomial}
The degree of the monomial ideal $M_\mathbf{a}$ is given by $\frac{1}{2}(\prod_{i=0}^n(a_i+1)-\prod_{i=0}^n (a_i-1))$.
\end{lemma}
\begin{proof}
We use induction on the number of variables. For $n=1$, $M_\mathbf{a}=(x_0^{a_0+a_1})$. Hence $\deg(M_\mathbf{a})$ is $a_0+a_1=\frac{1}{2}((a_0+1)(a_1+1)-(a_0-1)(a_1-1))$, which satisfies the formula. Now, suppose that for fewer variables, the statement is true. We first recall that one can have a generating set of $M_\mathbf{a}$ in which there is no generator having $x_n$ as mentioned above. Further, $x_k^{a_k+a_n}=\mathbf{x}^{\beta(\mathbf{a},\lambda_\mathbf{a}\mathbf{e}_k)}$ belongs to $M_\mathbf{a}$ for each $k=0,\ldots,n-1$. So, $T/M_\mathbf{a}$ admits a Stanley decomposition as 
$$T/M_\mathbf{a}=\bigoplus_{\mathbf{x}^\alpha\in N_\mathbf{a}} \mathbf{x}^\alpha~\mathbb{R}[x_n]~,$$
where $N_\mathbf{a}:=\{\mathbf{x}^\alpha\in\mathbb{R}[x_0,\ldots,x_{n-1}]  :  \mathbf{x}^\alpha\notin M_\mathbf{a}\}$ which is a finite set. By Proposition \ref{Stanley}, the Hilbert function of $T/M_\mathbf{a}$ can be calculated by 
$$\HF(T/M_\mathbf{a},t)=\sum_\alpha \HF(\mathbb{R}[x_n],t-|\alpha|)$$ and for $\forall~t\gg0$ Hilbert polynomial of $T/M_\mathbf{a}$ can be calculated as the constant $|N_\mathbf{a}|$. So the number of elements of $N_\mathbf{a}$ gives the degree of $M_\mathbf{a}$.
Now we divide the elements of $N_\mathbf{a}$ into pieces by exponent of $x_0$ and count the number of elements. Let 
\begin{align*}
M_{\mathbf{a},d}:=\{m\in \mathbb{R}[x_1,\ldots,x_n] : x_0^dm\in M_\mathbf{a}\}~\textrm{and}~N_{\mathbf{a},d}:=\{n\in \mathbb{R}[x_1,\ldots,x_{n-1}] : n\notin M_{\mathbf{a},d} \}~.
\end{align*}
Then $M_{\mathbf{a},d}$ is a monomial ideal in $\mathbb{R}[x_1,\ldots,x_n]$ and $N_\mathbf{a}$ is a disjoint union \\$\bigcup_{d=0}^{a_0+a_n-1} N_{\mathbf{a},d}\cdot x_0^d$, for $\mathbf{x}^{\beta(\mathbf{a},\lambda_\mathbf{a}\mathbf{e}_0)}=x_0^{a_0+a_n}\in M_{\mathbf{a}}$ (i.e. $M_{\mathbf{a},a_0+a_n}=\mathbb{R}[x_1,\ldots,x_n]$). So, we have
\begin{equation}\label{deg_Na}
|N_{\mathbf{a},d}|=\deg\big(\mathbb{R}[x_1,\ldots,x_n]/M_{\mathbf{a},d}\big) ~\textrm{and}~|N_\mathbf{a}|=\sum_{d=0}^{a_0+a_n-1}\deg\big(\mathbb{R}[x_1,\ldots,x_n]/M_{\mathbf{a},d}\big)~.
\end{equation}

When $i_0=0$, the monomial generator $\mathbf{x}^{\beta(\mathbf{a},\mathbf{i})}$ has $\beta_0=0$ (i.e. the monomial contains no $x_0$), while $\beta_0=a_0-1+2 i_0\ge a_0+1$ if $i_0\ge1$. For other $\beta_k$, if $i_0=0$, then it is given by 
$$\beta_k(\mathbf{a},\mathbf{i})=
\begin{cases}
0~, & i_k=0\\
a_k-1-2(i_1+\cdots+i_{k-1})+2i_k, & \text{otherwise~.}
\end{cases}$$
Hence, we see that $M_{\mathbf{a},d}=M_{(a_1,a_2,\ldots,a_n)}$ for $d=0,\ldots,a_0$. Similarly, if $i_0=j$ ($j=1,\ldots,\lambda_\mathbf{a}$), the generators of $M_\mathbf{a}$ is given by $\mathbf{x}^{\beta(\mathbf{a},\mathbf{i})}$ with 
$$\beta_k(\mathbf{a},\mathbf{i})=
\begin{cases}
0~, & i_k=0\\
(a_k-2j)-1-2(i_1+\cdots+i_{k-1})+2i_k, & \text{otherwise}
\end{cases}\quad,$$
which implies that $M_{\mathbf{a},d}=M_{(a_1-2j,a_2-2j,\ldots,a_n-2j)}$ for $a_0+2j-1\leq d\leq a_0+2j$. Therefore, by (\ref{deg_Na}) and induction hypothesis, we compute the degree of $M_\mathbf{a}$ as
{\small
\begin{align*}
|N_\mathbf{a}|&=\sum_{d=0}^{a_0} \deg(M_{(a_1,a_2,\ldots,a_n)})+\sum_{d=a_0+1}^{a_0+2}\deg(M_{(a_1-2,a_2-2,\ldots,a_n-2)})+\cdots+\sum_{d=a_0+a_n-2}^{a_0+a_n-1}\deg(M_{(a_1-2\lambda_\mathbf{a},a_2-2\lambda_\mathbf{a},\ldots,a_n-\lambda_\mathbf{a})})\\
&=(a_0+1)\deg(M_{(a_1,a_2,\ldots,a_n)})+2\deg(M_{(a_1-2,a_2-2,\ldots,a_n-2)})+\cdots+2\deg(M_{(a_1-a_n+1,a_2-a_n+1,\ldots,1)})\\
&=(a_0+1)\frac{1}{2}\big(\prod_{i=1}^n(a_i+1)-\prod_{i=1}^n(a_i-1)\big)+2\cdot\frac{1}{2}\big(\prod_{i=1}^n(a_i-1)-\prod_{i=1}^n(a_i-3)\big)+\cdots+2\cdot\frac{1}{2}(\prod_{i=1}^n(a_i-a_n+3)-0)\\
&=\frac{a_0+1}{2}\prod_{i=1}^n(a_i+1)-\frac{a_0+1}{2}\prod_{i=1}^n(a_i-1)+\prod_{i=1}^n(a_i-1)=\frac{1}{2}\prod_{i=0}^n(a_i+1)-\frac{1}{2}\prod_{i=0}^n(a_i-1)~,
\end{align*}}
as claimed.
\end{proof}

\begin{lemma}\label{prop_deg_drop}
For any monomial $\mathbf{n}\notin M_\mathbf{a}$, $\deg(M_\mathbf{a})>\deg(M_\mathbf{a}+(\mathbf{n}))$.
\end{lemma}
\begin{proof}
Let $\mathbf{n}$ be a monomial which is not contained in $M_\mathbf{a}$. It can be written as $\mathbf{n}=\mathbf{x}^{\gamma_0} x_n^k$ for some $\mathbf{x}^{\gamma_0}\in N_\mathbf{a}$ and $k\geq 0$. Set $\Gamma_0=\mathbb{R}[x_0,\ldots,x_{n-1}]\cdot\mathbf{x}^{\gamma_0}$ and $M'=M_\mathbf{a}+(\mathbf{n})$. Then, a Stanley decomposition of $T/M'$ can be given as
$$T/M'=\bigoplus_{\mathbf{x}^\alpha\in N_\mathbf{a}\backslash\Gamma_0}\mathbf{x}^\alpha~\mathbb{R}[x_n]~\oplus~  \bigoplus_{\mathbf{x}^\gamma\in N_\mathbf{a}\cap\Gamma_0}\left(\bigoplus_{i=0}^{k-1}\mathbf{x}^{\gamma}x_n^i~\mathbb{R}[\emptyset]\right)~,$$
which shows that the degree $\deg(M')$ is strictly less than $\deg(M_{_\mathbf{a}})$. 
\end{proof}

\begin{proposition}\label{initial supset}
For a sequence of positive integers $\mathbf{a}=(a_0,a_1,\ldots,a_n), ~a_0\geq a_1\geq \cdots\geq a_n$ with $a_n$ odd, $M_\mathbf{a}\subset in(J_\mathbf{a}(t))$ for all but finitely many real numbers $t$.
\end{proposition}

\begin{proof}
Let $\Supp(\mathbf{i})$ be the support of $\mathbf{i}\in \Delta^{n-1}(\lambda_\mathbf{a})$ and $\Supp^\#(\mathbf{i}):=\Supp(\mathbf{i})\cup\{n\}\subset\{0,1,\ldots,n\}$. We prove the statement by finding a polynomial $H(\mathbf{a},\mathbf{i})\in J_\mathbf{a}$ with the following three properties for any $\mathbf{i}\in \Delta^{n-1}(\lambda_\mathbf{a})$ in an inductive way. 
\begin{enumerate}
\item $\LT(H(\mathbf{a},\mathbf{i}))=A(\mathbf{a},\mathbf{i})\mathbf{x}^{\beta(\mathbf{a},\mathbf{i})}$, where $A(\mathbf{a},\mathbf{i})\in\R[t,t^{-1}]$ is a \textit{nonzero} leading coefficient.
\item Let $k\in \Supp^\#(\mathbf{i})$ be the smallest element. Then $H(\mathbf{a},\mathbf{i})\in \big(\{F_{k,k'}^\mathbf{a} : k'\in \Supp^\#(\mathbf{i}), k'\neq k\}\big)\subset J_\mathbf{a}$.
\item Moreover, $H(\mathbf{a},\mathbf{i})\in \mathbb{R}[t,t^{-1}][x_i,i\in \Supp^\#(\mathbf{i})]$
\end{enumerate}

To define $H(\mathbf{a},\mathbf{i})$, we use an induction on the size of $\Supp^\#(\mathbf{i})$ whose minimum is $2$. From now on, $H(\mathbf{i})$ stands for $H(\mathbf{a},\mathbf{i})$. We note that a nonzero element in $\R[t,t^{-1}]$ becomes nonzero by evaluating it at all but finitely many real numbers $t$. We also recall notations in Definition \ref{Gij}, \ref{beta} and Remark \ref{F rmk}. 
\vspace{2mm}

\noindent \textsc{Step 1 [Initial case I with $|\Supp^\#(\mathbf{i})|=2$]}

Suppose that $|\Supp^\#(\mathbf{i})|=2$. Then $\Supp(\mathbf{i})=\{k\}$ for some $0\leq k\leq n-1$ and $\mathbf{i}=\lambda_\mathbf{a}\mathbf{e}_k$. For this case, we find 
\begin{equation*}
H(\mathbf{i}):=G_{k,n}^\mathbf{a}=x_k^{a_k+a_n'}x_n^{\epsilon(a_n)}+C_{k,n,1}^{\mathbf{a}'}x_k^{a_k+a_n'-2}x_n^{2+\epsilon(a_n)}+\cdots=x_k^{a_k+a_n}+C_{k,n,1}^{\mathbf{a}'}x_k^{a_k+a_n-2}x_n^2+\cdots
\end{equation*} because $a_n$ is odd (here we briefly use notation $C_{k,n,d}^{\mathbf{a}'}$ instead of $C_{k,n,d}^{\mathbf{a}'}(t)$). Since $\mathbf{x}^{\beta(\mathbf{i})}=x_k^{a_k+a_n}$, $H(\mathbf{i})$ satisfies the properties (1) - (3). In case of $a_n=1$, everything is done in Step 1. So, in what follows, let us assume $a_n\ge3$.
\vspace{2mm}

\noindent\textsc{Step 2 [Initial case II with $|\Supp^\#(\mathbf{i})|=3$ and $\mathbf{i}=(\lambda_\mathbf{a}-1)\mathbf{e}_k+\mathbf{e}_l$]}

 Consider the another initial case $\mathbf{i}=(\lambda_\mathbf{a}-1)\mathbf{e}_k+\mathbf{e}_l$ for any $k<l$. Consider the division of $G_{k,l}^{\mathbf{a}}$ by $x_k^{a_k+a_n}$. Since they are homogeneous binary forms in $\mathbb{R}[t,t^{-1}][x_k,x_l]$, the division is same as the univariate case and the quotient is uniquely given. In particular, the quotient $q$ is given by $$q=\sum_{d=0}^{\frac{a_l'-a_n}{2}} C_{k,l,d}^{\mathbf{a}'}~ x_k^{a_l'-a_n-2d}x_l^{2d+\epsilon(a_l)}.$$ Define $H(\mathbf{i}):=G_{k,l}^{\mathbf{a}}-q\cdot G_{k,n}^{\mathbf{a}}$. Then, we see that 
 $$H(\mathbf{i})=\sum_{d=\frac{a_l'-a_n}{2}+1}^{\frac{a_k'+a_l'}{2}} C_{k,l,d}^{\mathbf{a}'}~ x_k^{a_k+a_l'-2d}x_l^{2d+\epsilon(a_l)}-q\cdot(G_{k,n}^{\mathbf{a}}-x_k^{a_k+a_n})~.$$
Since all the terms of the rear part involve $x_n$, the leading term should appear in the front. Hence $\LT(H(\mathbf{i}))=A(\mathbf{i})x_k^{a_k+a_n-2}x_l^{a_l-a_n+2}$ where $A(\mathbf{i})=C_{k,l,\frac{a_l'-a_n}{2}+1}^{\mathbf{a}'}$, which is nonzero. Note that $\mathbf{x}^{\beta(\mathbf{i})}=x_k^{a_k+a_n-2}x_l^{a_l-a_n+2}$. So, $H(\mathbf{i})$ satisfies the properties (1) - (3) of the polynomial $H$.

\vspace{2mm}

\noindent\textsc{Step 3 [Induction hypothesis and $\tau$-projection]}

Assume that $|\Supp^\#(\mathbf{i})|=m$ and all the polynomials $H(\mathbf{j})$ are already determined for all $\mathbf{j}$ with $|\Supp^\#(\mathbf{j})|<m$. Now, let us define a projection map $\tau$ on points in the simplex except vertices by $\tau(\mathbf{i})=\mathbf{i}+i_{l}\mathbf{e}_k-i_{l}\mathbf{e}_{l}$ where $k,l$ are the smallest index and the second smallest one in $\Supp^\#(\mathbf{i})$ respectively. Since $|\Supp^\#(\tau(\mathbf{i}))|=|\Supp^\#(\mathbf{i})\backslash\{l\}|=m-1$, we already have $H(\tau(\mathbf{i}))$ by the induction hypothesis. $H(\tau(\mathbf{i}))$ is a polynomial in $\mathbb{R}[t,t^{-1}][x_i,i\in \Supp^\#(\tau(\mathbf{i}))]$ where its leading term is 
$$\LT(H(\tau(\mathbf{i})))=A(\tau(\mathbf{i}))\mathbf{x}^{\beta(\tau(\mathbf{i}))}=A(\tau(\mathbf{i}))x_k^{a_k+2(i_k+i_l)-1}\mathbf{m}$$
where $\mathbf{m}$ is the remaining monomial part in $\mathbb{R}[x_i,i\in \Supp^\#(\mathbf{i})\backslash\{k,l\}]$. Note that $$\mathbf{x}^{\beta(\mathbf{i})}=x_k^{a_k+2i_k-1}x_l^{a_l-2i_k+2i_l-1}\mathbf{m}$$ for the same monomial $\mathbf{m}$ by Remark \ref{beta rmk}.
\vspace{2mm}

\noindent\textsc{Step 4 [$\mathsf{Move~I}$ - using two polynomials from Step 2 and 3]} 

In this step, using two polynomials $H\left((\lambda_\mathbf{a}-1)\mathbf{e}_k+\mathbf{e}_l\right)$ and $H(\tau(\mathbf{i}))$, we define $$H(\tau'(\mathbf{i})):=H(\tau(\mathbf{i})-\mathbf{e}_k+\mathbf{e}_l)~.$$ For the extreme case $i_k+i_l=\lambda_\mathbf{a}$, $\tau(\mathbf{i})$ is equal to $\lambda_\mathbf{a}\mathbf{e}_k$ so that $H(\tau'(\mathbf{i}))=H((\lambda_\mathbf{a}-1)\mathbf{e}_k+\mathbf{e}_l)$ which was already done in Step 2. Hence, suppose that $2\le i_k+i_l\leq \lambda_\mathbf{a}-1$. Write 
\begin{align}\label{Hpart}
\ds H\left((\lambda_\mathbf{a}-1)\mathbf{e}_k+\mathbf{e}_l\right)=\sum_{d=\frac{a_l'-a_n}{2}+1}^{\frac{a_k'+a_l'}{2}} C_{k,l,d}^{\mathbf{a}'} ~x_k^{a_k+a_l'-2d}x_l^{2d+\epsilon(a_l)}+R_1
\end{align}
where all terms in $R_1$ involving $x_n$ as in Step 2 and 
$H(\tau(\mathbf{i}))=A(\tau(\mathbf{i}))x_k^{a_k+2(i_k+i_l)-1}\mathbf{m}+R_2$ where $R_2$ is the remaining part except the leading term. Consider the division of the front part of $H\left((\lambda_\mathbf{a}-1)\mathbf{e}_k+\mathbf{e}_l\right)\in\mathbb{R}[t,t^{-1}][x_k,x_l]$ (i.e. terms except $R_1$ in (\ref{Hpart})) by $x_k^{a_k+2(i_k+i_l)-1}$ when regarding both dividend and divisor as a univariate polynomial in $x_k$. Then, the quotient $q'$ is given by $\ds q'=\sum_{d=d_{\min}}^{d_{\max}} C_{k,l,d}^{\mathbf{a}'} ~x_k^{2d_{\max}-2d}x_l^{2d+\epsilon(a_l)}$, where $d_{\min}=\frac{a_l'-a_n}{2}+1, d_{\max}=\frac{(a_k+a_l')-(a_k+2(i_k+i_l)-1)}{2}=\frac{a_l'-2(i_k+i_l)+1}{2}$.
Using this quotient, define 
\begin{align}\label{Htau'}\nonumber
H(\tau'(\mathbf{i})):&=A(\tau(\mathbf{i}))\mathbf{m}\cdot H\left((\lambda_\mathbf{a}-1)\mathbf{e}_k+\mathbf{e}_l\right)-q'\cdot H(\tau(\mathbf{i}))\\
&=A(\tau(\mathbf{i}))\mathbf{m}\cdot\left(\sum_{d=d_{\max}+1}^{\frac{a_k'+a_l'}{2}}C_{k,l,d}^{\mathbf{a}'}~x_k^{a_k+a_l'-2d}x_l^{2d+\epsilon(a_l)}+R_1\right)-q'\cdot R_2~.
\end{align}
Since $x_k^{a_k+2(i_k+i_l)-1}\mathbf{m}\succ_{grevlex}$ any term in $R_2$, $x_k^{a}x_l^{b}\mathbf{m}$ is always larger than any term in $q'\cdot R_2$ whenever both have the same total degree. Thus, the leading term of $H(\tau'(\mathbf{i}))$ should appear in the front part of (\ref{Htau'}) and $\ds\LT(H(\tau'(\mathbf{i})))=A(\tau(\mathbf{i}))\mathbf{m}\cdot C_{k,l,d_{\max}+1}^{\mathbf{a}'}~x_k^{a_k+2(i_k+i_l)-3}x_l^{a_l-2(i_k+i_l)+3}$.
Since the leading coefficient \\$A(\tau'(\mathbf{i}))=A(\tau(\mathbf{i}))C_{k,l,d_{\max}+1}^{\mathbf{a}'}$ is nonzero in $\R[t,t^{-1}]$ and
$$\mathbf{x}^{\beta(\tau'(\mathbf{i}))}=x_k^{a_k+2(i_k+i_l-1)-1}x_l^{a_l-2(i_k+i_l-1)+2-1}\mathbf{m}$$
with the same monomial $\mathbf{m}$ as above by Remark \ref{beta rmk}, $H(\tau'(\mathbf{i}))$ satisfies property (1). It is easy to check that it also has the other properties (2) - (3) as well. 
\vspace{2mm}

\noindent\textsc{Step 5 [$\mathsf{Move~II}$ - using two polynomials from Step 3 and 4]} 

In the final step, using $H(\tau(\mathbf{i}))$ and $H(\tau'(\mathbf{i}))$, we define $H(\tau(\mathbf{i})-s\mathbf{e}_k+s\mathbf{e}_l)$ for any $s=2,\ldots,i_k+i_l-1$. In particular, if we put $s=i_l$, then $H(\tau(\mathbf{i})-i_l\mathbf{e}_k+i_l\mathbf{e_l})=H(\mathbf{i})$ and the proof can be completed by the induction.

Recall that from \textsc{Step 3} and \textsc{4} we can write $H(\tau(\mathbf{i}))=A(\tau(\mathbf{i}))x_k^{a_k+2(i_k+i_l)-1}\mathbf{m}+R_2$ and
$$H(\tau'(\mathbf{i}))=A(\tau(\mathbf{i}))\sum_{d=d_{\max}+1}^{\frac{a_k'+a_l'}{2}}C_{k,l,d}^{\mathbf{a}'}~x_k^{a_k+a_l'-2d}x_l^{2d+\epsilon(a_l)}\mathbf{m}+R_3~.$$
Consider two homogeneous binary forms $$f=x_k^{a_k+2(i_k+i_l)-1}
\qquad \textrm{and} \qquad g=\sum_{d=d_{\max}+1}^{\frac{a_k'+a_l'}{2}}C_{k,l,d}^{\mathbf{a}'}~x_k^{a_k+a_l'-2d}x_l^{2d+\epsilon(a_l)}~.$$ To simplify the situation, divide them by $x_k^{\epsilon(a_k)}$, substitute $x_k^2=X$ and dehomogenize them by putting $x_l=1$. Then we have $\ds\tilde{f}=X^\alpha, ~\tilde{g}=\sum_{i=0}^{\alpha-1}D_iX^i$ where $D_i=C_{k,l,\frac{a_k'+a_l'}{2}-i}^{\mathbf{a}'}$ if we put $\alpha:=\frac{a_k'+2(i_k+i_l)-1}{2}$. Now, since $\tilde{f}$ and $\tilde{g}$ are univariate polynomials with degree $\alpha$ and $\alpha-1$ respectively, we can apply Lemma \ref{prs}. Then, for $i=\alpha-s+1$ ($2\leq s\leq i_k+i_l-1$), there exist two polynomials $\tilde{u}_{\alpha-s+1},\tilde{v}_{\alpha-s+1}$ such that $\tilde{h}=\tilde{u}_{\alpha-s+1}\tilde{f}+\tilde{v}_{\alpha-s+1}\tilde{g}$ have degree less then or equal to $\alpha-s$. Further, each coefficient of $\tilde{h}$ is given by the determinant of a corresponding submatrix of the Sylvester matrix of $\tilde{f}$ and $\tilde{g}$ as in (\ref{uf+vg}). In this case, the corresponding submatrix is given by the $(\alpha+s-1)\times (2s-1)$ matrix
$$\small{M=
\begin{pmatrix*}[l]
1  & 0 & \cdots & 0			&	D_{\alpha-1} & 0 & \cdots & 0					\\
0 & 1 & \cdots & 0			&	D_{\alpha-2} & D_{\alpha-1} & \cdots & 0					\\
\vdots & \vdots & \ddots & \vdots&	\vdots & \vdots & \ddots & \vdots					\\
0 & 0 & \cdots & 1			&	D_{\alpha-s+1} & D_{\alpha-s+2} & \cdots & 0					\\
0 & 0 & \cdots & 0			&	D_{\alpha-s} & D_{\alpha-s+1} & \cdots & D_{\alpha-1}					\\
\vdots & \vdots & \ddots & \vdots&		\vdots & \vdots & \ddots & \vdots				\\
0 & 0 & \cdots & 0 			&	0 & 0& \cdots & D_0
\end{pmatrix*}.}
$$
In particular, the leading coefficient of $\tilde{h}$ is equal to the determinant of $\widetilde{M}_{0}$, the top $(2s-1)\times(2s-1)$ submatrix of $M$,  by (\ref{uf+vg}). Now, we show that $\deg\tilde{h}=\alpha-s$ (i.e. $\det(\widetilde{M}_{0})$ is nonzero in $\R[t,t^{-1}]$). Since the top left $(s-1)\times(s-1)$ submatrix of $\widetilde{M}_{0}$ is the identity matrix and only zero entries in below, it is enough to calculate the determinant of the bottom right $s\times s$ submatrix such as
$$\small{N:=\begin{pmatrix*}[l]
D_{\alpha-s} & D_{\alpha-s+1} & \cdots & D_{\alpha-1} \\
D_{\alpha-s-1} & D_{\alpha-s} & \cdots & D_{\alpha-2} \\
\vdots & \vdots & \ddots & \vdots\\
D_{\alpha-2s+1} & D_{\alpha-2s+2} & \cdots & D_{\alpha-s}
\end{pmatrix*}~.}
$$
We know that the top degree part of $D_i\in\R[t,t^{-1}]$ has degree $Q_{k,l}^{\mathbf{a}'}(\frac{a_k'+a_l'}{2}-i)$, which is $-i^2+\gamma_1 i +\gamma_2$ for some constant $\gamma_1, \gamma_2$ (see Remark \ref{F rmk} (3)). And each term of the determinant of $N$ is of the form $D_{i_1}D_{i_2}\cdots D_{i_s}$ with $i_1+i_2+\cdots+i_s=s(\alpha-s)$. The top degree of this term is given by
\begin{equation}\label{top_det_N}
\sum_{i=i_1,\ldots,i_s}(-i^2+\gamma_1 i +\gamma_2)=\sum_{i=i_1,\ldots,i_s}(-i^2)+\gamma_1 s(\alpha-s) +\gamma_2 s ~.
\end{equation}
Since the sum $i_1+i_2+\cdots+i_s$ is fixed, the maximum of (\ref{top_det_N}) is achieved only when all $i_1=i_2=\cdots=i_s=\alpha-s$. Hence the top degree part of $\det(N)$ comes only from the diagonal term $D_{\alpha-s}^s$ which is not identically zero in $\R[t,t^{-1}]$. Therefore, the leading coefficient of $\tilde{h}$, $\det(N)$ is non-zero and $\deg\tilde{h}=\alpha-s$. Further, by the proof of Lemma \ref{prs} the leading coefficients of $\tilde{u}_{\alpha-s+1}$ and $\tilde{v}_{\alpha-s+1}$ are given up to sign by $D_{\alpha-1}\cdot\det(\overline{N})$ and $\det(\overline{N})$ respectively, where $\det(\overline{N})$ is determinant of the top right $(s-1)\times(s-1)$ submatrix of $N$ which is nonzero in $\R[t,t^{-1}]$ in a similar manner as for $\det(N)$. So, we also have $\deg\tilde{u}_{\alpha-s+1}=s-2$ and $\deg\tilde{v}_{\alpha-s+1}=s-1$.

Let $u,v,h$ be the homogenization of $\tilde{u}_{\alpha-s+1},\tilde{v}_{\alpha-s+1},\tilde{h}$ by $x_l$ and substitute $X$ by $x_k^2$, respectively. Then we have $u f x_l^{a_l-2(i_k+i_l)+3}+vg=h x_k^{\epsilon(a_k)}x_l^{a_l-2(i_k+i_l)+4s-1}$ and $h$ is a degree $2(\alpha-s)$ homogeneous binary form in $\mathbb{R}[t,t^{-1}][x_k,x_l]$ with $\LT(h x_k^{\epsilon(a_k)})=\det(N)x_k^{2\alpha-2s}x_k^{\epsilon(a_k)}=\det(N)x_k^{a_k+2(i_k+i_l-s)-1}$. Now let us set
\begin{align*}
H(\tau(\mathbf{i})-s\mathbf{e}_k+s\mathbf{e}_l)&:=u\cdot x_l^{a_l-2(i_k+i_l)+3}\cdot H(\tau(\mathbf{i}))+v\cdot H(\tau'(\mathbf{i}))\\
&=u\cdot x_l^{a_l-2(i_k+i_l)+3}(A(\tau(\mathbf{i}))f\mathbf{m}+R_2)+v\cdot(A(\tau(\mathbf{i}))g\mathbf{m}+R_3)\\
&=A(\tau(\mathbf{i}))(u f x_l^{a_l-2(i_k+i_l)+3}+vg)\mathbf{m}+ux_l^{a_l-2(i_k+i_l)+3}R_2+v R_3\\
&=A(\tau(\mathbf{i}))h x_k^{\epsilon(a_k)}x_l^{a_l-2(i_k+i_l)+4s-1}\mathbf{m}+R_4~.
\end{align*}

Similarly, the leading term appears in the front part and it is given by
$$\LT(H(\tau(\mathbf{i})-s\mathbf{e}_k+s\mathbf{e}_l))=A(\tau(\mathbf{i}))\det(N)\cdot x_k^{a_k+2(i_k+i_l-s)-1}x_l^{a_l-2(i_k+i_l-s)+2s-1}\mathbf{m}$$
while $\mathbf{x}^{\beta(\tau(\mathbf{i})-s\mathbf{e}_k+s\mathbf{e}_l)}$ is $x_k^{a_k+2(i_k+i_l-s)-1}x_l^{a_l-2(i_k+i_l-s)+2s-1}\mathbf{m}$ by Remark \ref{beta rmk}. So, $H(\tau(\mathbf{i})-s\mathbf{e}_k+s\mathbf{e}_l)$ satisfies property (1) and (2)-(3) is also immediate from the construction.
\end{proof}

We present some examples illustrating the process in the proof of Proposition \ref{initial supset} concretely.
 
\begin{example}
Let $n=3$ and $\mathbf{a}=(5,5,5,5)$, $\lambda_\mathbf{a}=\frac{5+1}{2}=3$. Take $t=2$.\\
Then the simplex $\Delta^{3-1}(\lambda_\mathbf{a})=\Delta^2(3)$ consists of the following $10$ integer lattice points 
$$(3,0,0),(2,1,0),(1,2,0),(0,3,0),(2,0,1),(1,1,1),(0,2,1),(1,0,2),(0,1,2),(0,0,3).$$
From these points, the monomial ideal $M_\mathbf{a}$ is generated by 10 monomials
$$(x_0^{10},x_0^8x_1^2,x_0^6x_1^6,x_1^{10},x_0^8x_2^2,x_0^6x_1^4x_2^2,x_1^8x_2^2,x_0^6x_2^6,x_1^6x_2^6,x_2^{10})$$
To follow the proof of Proposition \ref{initial supset}, we assign a polynomial $H(\mathbf{i})\in J_{\mathbf{a}}(t)$ for each lattice point $\mathbf{i}$ such that the leading monomial of $H(\mathbf{i})$ is $\mathbf{x}^{\beta(\mathbf{a},\mathbf{i})}$, a generator of $M_\mathbf{a}$. 

First, $H(3,0,0), H(2,1,0),H(0,3,0),H(2,0,1),H(0,2,1)$ and $H(0,0,3)$ can be obtained by the \textsc{Step 1} and \textsc{Step 2}. To get $H(1,1,1)$, we need to apply \textsf{Move I} in \textsc{Step 4} on the polynomials $H(2,1,0)$ and $H(2,0,1)$. They are given by the following form
\begin{align*}
&H(2,1,0)=F_{0,1}^\mathbf{a}-F_{0,3}^\mathbf{a}=-\frac{341}{16}x_0^8x_1^2+\frac{5797}{64}x_0^6x_1^4+\cdots-x_1^{10}+R_1\\
&H(2,0,1)=F_{0,2}^\mathbf{a}-F_{0,3}^\mathbf{a}=-\frac{341}{16}x_0^8x_2^2+\frac{5797}{64}x_0^6x_2^4+\cdots-x_2^{10}+R_2
\end{align*}
by the \textsc{Step 2} (note that $\beta((2,1,0))=(5+2\cdot2-1,5-2\cdot2+2\cdot1-1,0)=(8,2,0)$ and $\beta((2,0,1))=(5+2\cdot2-1,0,5-2\cdot2+2\cdot1-1)=(8,0,2)$, as desired). To apply \textsf{Move I} in \textsc{Step 4}, we need to calculate $q'=-341/16x_1^2$. Using this, we obtain 
\begin{align*}
H(1,1,1)&=A(2,0,1)\mathbf{m}\cdot H(2,1,0)-q'H(2,0,1)\\
&=-\frac{341}{16}x_2^2H(2,1,0)+\frac{341}{16}x_1^2H(2,0,1)=A(1,1,1)x_0^6x_1^4x_2^2+\frac{1976777}{1024}x_0^4x_1^6x_2^2+\cdots
\end{align*}
Since $\beta((1,1,1))=(5+2\cdot1-1,5-2\cdot1+2\cdot1-1,5-2\cdot(1+1)+2\cdot1-1)=(6,4,2)$, it satisfies the properties of the polynomial $H$. 

Finally, to construct $H(1,2,0)$, we're going to apply \textsf{Move II} with $s=2$ on the polynomials $H(3,0,0)$ and $H(2,1,0)$. $H(3,0,0)=F_{0,3}^\mathbf{a}=x_0^{10}+\cdots$ and $H(2,1,0)$ is given as above. Then, from these polynomials \textsf{Move II} with $s=2$ produces $H(1,2,0)$ as follows
\begin{align*}
H(1,2,0)&=(-\frac{341}{16})^2x_1^2\cdot H(3,0,0)+(\frac{341}{16}x_0^2+\frac{5797}{64}x_1^2)\cdot H(2,1,0)\\
&=\frac{25698101}{4096}x_0^6x_1^6-\frac{31744713}{4096}x_0^4x_1^8+\cdots.
\end{align*}
Note that $\beta((1,2,0))=(5+2\cdot 1-1,5-2\cdot 1+2\cdot 2-1,0)=(6,6,0)$ with the properties for $H$.
\end{example}

\begin{example}
We sketch a general scheme for $H$ polynomials as exemplifying more complicate cases. For brevity, we denote the \textsf{Move I} in \textsc{Step 4} by $A_1$ and \textsf{Move II} in \textsc{Step 5} with $s$ by $A_2(s)$. 
\begin{enumerate}
\item[(a)] Let $n=3$ and $\mathbf{a}=(17,15,11,9)$. The simplex $\Delta^{3-1}(\frac{9+1}{2})=\Delta^2(5)$ is consists of $21$ lattice points. If we want to obtain the polynomial $H(2,2,1)$, we can follow steps as below. 
\begin{enumerate}
\item[(i)] Calculate $H(4,0,1)$ and $H(4,1,0)$ by \textsc{Step 2}.
\item[(ii)] Applying the \textsf{Move I} in \textsc{Step 4} with $H(4,0,1)$ and $H(4,1,0)$, we obtain $H(3,1,1)$.
$$H(4,0,1)+H(4,1,0)\xrightarrow{A_1} H(3,1,1)$$
\item[(iii)] Applying the \textsf{Move II} in \textsc{Step 5} with $H(4,0,1)$ and $H(3,1,1)$, we obtain $H(2,2,1)$.
$$H(4,0,1)+H(3,1,1)\xrightarrow{A_2(2)} H(2,2,1)$$
\end{enumerate}


\item[(b)] Let $n=4$ and $\mathbf{a}=(74,68,64,55,49)$. In case of finding the polynomial $H(3,9,6,7)$, one can do that as the following steps.
\begin{align*}
&H(25,0,0,0)+H(24,0,0,1)\xrightarrow{A_2(7)} H(18,0,0,7)~,~ H(18,0,0,7)+H(24,0,1,0)\xrightarrow{A_1} H(17,0,1,7)~,\\
&H(18,0,0,7)+H(17,0,1,7)\xrightarrow{A_2(6)} H(12,0,6,7)~,~ H(12,0,6,7)+H(24,1,0,0)\xrightarrow{A_1} H(11,1,6,7)~,\\
&H(12,0,6,7)+H(11,1,6,7)\xrightarrow{A_2(9)} H(3,9,6,7)~.
\end{align*}
\end{enumerate}
\end{example}

As mentioned at the beginning of the subsection, now we show that $M_\mathbf{a}$ coincides with $in(J_\mathbf{a})$ for the case of $a_n$ odd.

\begin{proposition}\label{in_J_a_odd}
For a sequence of positive integers $\mathbf{a}=(a_0,a_1,\ldots,a_n),~a_0\ge\ldots\ge a_n$ with $a_n$ odd, the initial ideal $in(J_\mathbf{a}(t))$ is same as $M_\mathbf{a}$ for all but finitely many real number $t$.
\end{proposition}

\begin{proof}
First of all, we know $in(J_\mathbf{a}(t))\supset M_\mathbf{a}$ for all real $t$ with finite exceptions by Proposition \ref{initial supset}. Also, by Proposition \ref{Ja} and Lemma \ref{degree of monomial} we have 
$$deg(J_\mathbf{a})\geq |V(J_\mathbf{a})|=\frac{1}{2}\left(\prod_{i=0}^n(a_i+1)-\prod_{i=0}^n(a_i-1)\right)=\deg(M_\mathbf{a})~.$$ 
Suppose that $in(J_\mathbf{a})\neq M_\mathbf{a}$. Then there exists a polynomial $f\in J_\mathbf{a}$ with $\mathbf{n}:=\LT(f)\notin M_\mathbf{a}$ so that
$in(J_\mathbf{a})\supset M_\mathbf{a}+(\mathbf{n})\supset M_\mathbf{a}$.
Since both $in(J_\mathbf{a})$ and $M_\mathbf{a}$ are 0-dimensional ideals, by Lemma \ref{prop_deg_drop}, we derive
$$\deg(J_\mathbf{a})=\deg(in(J_\mathbf{a}))\leq \deg(M_\mathbf{a}+(\mathbf{n}))<\deg(M_\mathbf{a})~,$$ which is a contradiction. Therefore, $in(J_\mathbf{a})=M_\mathbf{a}$. 
\end{proof}

Next, we describe $in(J_\mathbf{a})$ for a sequence $\mathbf{a}$ with even $a_n$.

\begin{proposition}\label{in_J_a_even}
Let $\mathbf{a}$ be a sequence $\mathbf{a}=(a_0,a_1,\ldots,a_n),~a_0\ge\ldots\ge a_n$ as above. If $a_n$ is even, then for all but finitely many real number $t$ the initial ideal $in(J_\mathbf{a}(t))$ can be generated by
$$in(J_{(a_0,a_1,\ldots,a_{n-1})})\cup x_n \cdot in(J_{\mathbf{a}-\mathbf{e}_n})~.$$
\end{proposition}
\begin{proof}
First, we show that $in(J_\mathbf{a})\supset in(J_{(a_0,a_1,\ldots,a_{n-1})})\cup x_n\cdot  in(J_{\mathbf{a}-\mathbf{e}_n})$. It is easy to see that $in(J_\mathbf{a})\supset in(J_{(a_0,a_1,\ldots,a_{n-1})})$ since $J_\mathbf{a}\supset J_{(a_0,a_1,\ldots,a_{n-1})}$. Also, for any $f\in J_{\mathbf{a}-\mathbf{e}_n}$, it can be written as $$f=\sum_{0\leq i<j\leq n} h_{i,j}G_{i,j}^{\mathbf{a}-\mathbf{e}_n}$$ for some $h_{i,j}$ in $T$. Since $G_{i,j}^\mathbf{a}=G_{i,j}^{\mathbf{a}-\mathbf{e}_n}$ for any $0\le i<j<n$ and $G_{i,n}^\mathbf{a}=x_n G_{i,n}^{\mathbf{a}-\mathbf{e}_n}$ (recall Definition \ref{Gij}), we see $x_n f\in J_{\mathbf{a}}$ so that $x_n\cdot J_{\mathbf{a}-\mathbf{e}_n}\subset J_\mathbf{a}$. Hence, we verify that $x_n\cdot  in(J_{\mathbf{a}-\mathbf{e}_n})\subset in(J_\mathbf{a})$. 

For the other inclusion $in(J_\mathbf{a})\subset \langle in(J_{(a_0,a_1,\ldots,a_{n-1})})\cup x_n\cdot  in(J_{\mathbf{a}-\mathbf{e}_n})\rangle$, we need to prove that for any element $f\in J_\mathbf{a}$, $LT(f)\in in(J_{(a_0,a_1,\ldots,a_{n-1})})$ or $LT(f)\in x_n\cdot  in(J_{\mathbf{a}-\mathbf{e}_n})$. In a similar manner as above, we note that one can write $f$ as 
\begin{equation}\label{f_expr}
f=\sum_{0\leq i<j\leq n} g_{i,j}G_{i,j}^{\mathbf{a}}=\sum_{0\leq i<j<n} g_{i,j}G_{i,j}^\mathbf{a}+\sum_{0\leq i\leq n-1} g_{i,n} G_{i,n}^\mathbf{a}=\sum_{0\leq i<j<n} g_{i,j}G_{i,j}^\mathbf{a}+x_n\cdot\sum_{0\leq i\leq n-1} g_{i,n} G_{i,n}^{\mathbf{a}-\mathbf{e}_n}
\end{equation}
for some $g_{i,j}$ in $T$. Since we use the graded reverse lexicographic order, if the front part of (\ref{f_expr}) (i.e. the sum over $0\leq i<j<n$) is nonzero, the leading term of $f$ is determined only by the front, which corresponds to $LT(f)\in in(J_{(a_0,a_1,\ldots,a_{n-1})})$. On the other hand, if the front is zero, the leading term of $f$ should be come from the rear of (\ref{f_expr}). In the latter case, $LT(f)\in x_n\cdot  in(J_{\mathbf{a}-\mathbf{e}_n})$. 
\end{proof}
\begin{proposition}\label{degree of ideal}
For a sequence of positive integers $\mathbf{a}=(a_0,\ldots,a_n)$ with $a_0\geq a_1\geq \cdots \geq a_n$, the degree of $J_\mathbf{a}$ is given by
$$\deg(J_\mathbf{a})=\frac{1}{2}(\prod_{i=0}^n (a_i+1)-\prod_{i=0}^n (a_i-1))~.$$
\end{proposition}
\begin{proof}
We will use induction on $n+1$, the number of variables. For the initial case, $n=1$, $in(J_{(a_0,a_1)})=(x_0^{a_0+a_1})$ or $in(J_\mathbf{a})=(x_0^{a_0+a_1-1}x_1)$ depending on the parity of $a_1$. For both cases, $\deg(J_{(a_0,a_1)})=a_0+a_1$ and it satisfies the assertion. Now, suppose that for fewer number of variables the claim does hold.
If $a_n$ is odd, by Proposition \ref{in_J_a_odd}, 
\begin{align*}
\deg(J_\mathbf{a})&=\deg(in(J_\mathbf{a}))=\deg(M_\mathbf{a})=\frac{1}{2}(\prod_{i=0}^n (a_i+1)-\prod_{i=0}^n (a_i-1))~.
\end{align*}
In case of $a_n$ even, by Proposition \ref{in_J_a_even}, we have
\begin{align*}
in(J_\mathbf{a})&=\langle in(J_{(a_0,a_1,\ldots,a_{n-1})}) \cup (x_n)\cap in(J_{(a_0,a_1,\ldots,a_{n}-1)})\rangle\\
&=in(J_{(a_0,a_1,\ldots,a_{n-1})}) + (x_n)\cap in(J_{(a_0,a_1,\ldots,a_{n}-1)})=\left(in(J_{(a_0,a_1,\ldots,a_{n-1})})+(x_n)\right) \cap in(J_{(a_0,a_1,\ldots,a_{n}-1)})
\end{align*}
since $in(J_{(a_0,a_1,\ldots,a_{n-1})})\subset in(J_{(a_0,a_1,\ldots,a_{n}-1)})=M_{(a_0,a_1,\ldots,a_{n}-1)}$. 

Put $A=in(J_{(a_0,a_1,\ldots,a_{n-1})})+(x_n)$ and $B=in(J_{(a_0,a_1,\ldots,a_{n}-1)})=M_{(a_0,a_1,\ldots,a_{n}-1)}$. Then, $in(J_\mathbf{a})=A\cap B$. For $A+B\supset (x_n)+M_{(a_0,a_1,\ldots,a_{n}-1)}$ and $x_i^{a_i+a_n-1}\in M_{(a_0,a_1,\ldots,a_{n}-1)}$ for all $i=0,\ldots,n-1$, we can observe that $V(A+B)=\emptyset$. So, from the exact sequence 
$$0\rightarrow T/(A\cap B)\rightarrow T/A\oplus T/B\rightarrow T/(A+B)\rightarrow 0$$
we derive that for any $t\gg0$
\begin{align*}
\HF(T/(A\cap B),t)&=\HF(T/A,t)+\HF(T/B,t)\\
&=\HF(T/(in(J_{(a_0,a_1,\ldots,a_{n-1})})+(x_n)),t)+\HF(T/M_{(a_0,a_1,\ldots,a_{n}-1)},t)\\
&=\HF(\mathbb{R}[x_0,x_1,\ldots,x_{n-1}]/in(J_{(a_0,a_1,\ldots,a_{n-1})}),t)+\HF(T/M_{(a_0,a_1,\ldots,a_{n}-1)},t)~\cdots(\ast).
\end{align*}
By the induction hypothesis, we have
$$\HF(\mathbb{R}[x_0,x_1,\ldots,x_{n-1}]/in(J_{(a_0,a_1,\ldots,a_{n-1})}),t)=\deg(J_{(a_0,a_1,\ldots,a_{n-1})})=\frac{1}{2}(\prod_{i=0}^{n-1}(a_i+1)-\prod_{i=0}^{n-1}(a_i-1))$$
and since $a_n-1$ is odd, we also obtain
$$\HF(T/M_{(a_0,a_1,\ldots,a_{n}-1)},t)=\deg(J_{(a_0,a_1,\ldots,a_{n}-1)})=\frac{1}{2}((a_n)\prod_{i=0}^{n-1}(a_i+1)-(a_n-2)\prod_{i=0}^{n-1}(a_i-1))~.$$
Using $(\ast)$ and two equalities above, we conclude
\begin{align*}
\deg(J_\mathbf{a})&=\deg(in(J_\mathbf{a}))=\deg(A\cap B)=\HF(T/(A\cap B),t)\quad\textrm{ for any $t\gg0$}\\
&=\frac{1}{2}(\prod_{i=0}^{n-1}(a_i+1)-\prod_{i=0}^{n-1}(a_i-1))+\frac{1}{2}((a_n)\prod_{i=0}^{n-1}(a_i+1)-(a_n-2)\prod_{i=0}^{n-1}(a_i-1))\\
&=\frac{1}{2}((a_n+1)\prod_{i=0}^{n-1}(a_i+1)+(a_n-1)\prod_{i=0}^{n-1}(a_i-1))=\frac{1}{2}(\prod_{i=0}^n(a_i+1)-\prod_{i=0}^n(a_i-1))~.
\end{align*}
\end{proof}

\begin{remark}\label{vars_inJa}
We make some remarks on the initial ideal $in(J_\mathbf{a})$.
\begin{enumerate}
\item By Proposition \ref{in_J_a_even} and a typical induction argument, we obtain a fact that $in(J_{(a_0,a_1,\ldots,a_k)})$ can be generated by monomials concerning only the variables $x_0,\ldots,x_k$ no matter how $a_k$ is odd or even. Note that the fact is true both for $k=1$ because $in(J_{(a_0,a_1)})=(x_0^{a_0+a_1'}x_1^{\epsilon(a_1)})$ and for any odd $a_k$ by Definition \ref{def_Ma}.
\item If there is an odd index $a_i$ for some $i\ge1$ in a given $\mathbf{a}=(a_0,a_1,\ldots,a_n)$, let $a_k$ be the first odd index  from the last. Then, as applying Proposition \ref{in_J_a_even} recursively, we have
\begin{align*}
in(J_\mathbf{a})&=\langle in(J_{(a_0,a_1,\ldots,a_{n-1})})\cup x_n \cdot in(J_{(a_0,a_1,\ldots,a_{n}-1)})\rangle\\
&=\langle in(J_{(a_0,\ldots,a_{n-2})})\cup x_{n-1} \cdot in(J_{(a_0,\ldots,a_{n-1}-1)})\cup x_n \cdot in(J_{(a_0,\ldots,a_{n}-1)})\rangle \\
&\cdots\\
&=\langle in(J_{(a_0,\ldots,a_{k})})\cup x_{k+1} \cdot in(J_{(a_0,\ldots,a_{k+1}-1)})\cup \cdots \cup x_n \cdot in(J_{(a_0,\ldots,a_{n}-1)})\rangle \\
&=\left\langle M_{(a_0,\ldots,a_{k})} \cup \bigcup_{i=k+1}^{n} x_i\cdot M_{(a_0,\ldots,a_{i}-1)}\right\rangle~.
\end{align*}
Similarly, in case of all $a_i$ being even for $i\ge1$, we also obtain
\begin{align*}
in(J_\mathbf{a})&=\langle in(J_{(a_0,a_{1})}) \cup \bigcup_{i=2}^{n} x_i\cdot M_{(a_0,\ldots,a_{i}-1)}\rangle =\left\langle(x_0^{a_0+a_1-1}x_1) \cup \bigcup_{i=2}^{n} x_i\cdot M_{(a_0,\ldots,a_{i}-1)}\right\rangle~.
\end{align*}
\end{enumerate}
\end{remark}

\subsection{Proof for the upper bound}\label{subsect_3-3}

Using structures of the initial ideal $in(J_\mathbf{a})$, we first show that the ideal $J_\mathbf{a}$ is saturated. 

\begin{proposition}\label{satu}
Let $\mathbf{a}$ be a sequence of positive integers $(a_0,a_1,\ldots,a_n), ~a_0\geq a_1\geq \cdots\geq a_n$. For all but finitely many real number $t$, the ideal $J_{\mathbf{a}}(t)$ is saturated in $T$.\end{proposition}
\begin{proof}
Denote $J_{\mathbf{a}}(t)$ just by $J_{\mathbf{a}}$. Recall that $J_\mathbf{a}$ is saturated in $T$ if and only if for any polynomial $g$ with $gx_0^{m_0}, g x_1^{m_1},\ldots, g x_n^{m_n}\in J_\mathbf{a}$ for some $m_i$ in $\mathbb{Z}_{>0}$, $g$ belongs to $J_\mathbf{a}$. We prove the statement by contradiction.

Suppose that $g$ is not contained in $J_\mathbf{a}$. Let $r$ be the remainder in the division of $g$ by the ideal $J_\mathbf{a}$ with respect to the graded reverse lexicographic order. Then it holds that $\LT(r)\notin in(J_\mathbf{a})$ and $\LT(r)x_i^{m_i}\in in(J_\mathbf{a})$ for $i=0,1,\ldots,n$. Let us begin with the condition $\LT(r)x_n^{m_n}\in in(J_\mathbf{a})$.

First, if $a_n$ is odd, by Proposition \ref{in_J_a_odd} $in(J_\mathbf{a})=M_\mathbf{a}$ and as in Definition \ref{def_Ma} there is a generating set for $M_\mathbf{a}$ which does not have any generator with $x_n$ in it. It means that $\LT(r)x_n^{m_n}\in M_\mathbf{a}$ implies $\LT(r)\in M_\mathbf{a}=in(J_\mathbf{a})$, which is a contradiction. On the other hand, if $a_n$ is even, by Proposition \ref{in_J_a_even} we have 
$in(J_\mathbf{a})=\langle in(J_{(a_0,a_1,\ldots,a_{n-1})})\cup x_n\cdot in(J_{(a_0,a_1,\ldots,a_n-1)})\rangle$.  
Thus, $\LT(r)x_n^{m_n}\in in(J_{(a_0,a_1,\ldots,a_{n-1})})$ or $\LT(r)x_n^{m_n}\in x_n\cdot in(J_{(a_0,a_1,\ldots,a_n-1)})$. For the former case, because of generators of $in(J_{(a_0,a_1,\ldots,a_{n-1})})$ with no $x_n$ (see Remark \ref{vars_inJa} (1)), similarly we get 
$$\LT(r)\in in(J_{(a_0,a_1,\ldots,a_{n-1})})\subset in(J_\mathbf{a})$$ which is a contradiction. For the latter case, since $a_n-1$ is odd, we have \\$in(J_{(a_0,a_1,\ldots,a_n-1)})=M_{(a_0,a_1,\ldots,a_n-1)}$ and know that the generators of $M_{(a_0,a_1,\ldots,a_n-1)}$ do not involve $x_n$. Thus, $\LT(r)x_n^{m_n}\in x_n\cdot in(J_{(a_0,a_1,\ldots,a_n-1)})$ means that $\LT(r)\in M_{(a_0,a_1,\ldots,a_n-1)}$. Let $\LT(r):=x_0^{\gamma_0}x_1^{\gamma_1}\cdots x_n^{\gamma_n}$. Then, we claim that $\gamma_n=0$. If $\gamma_n>0$, then $\LT(r)\in x_n M_{(a_0,a_1,\ldots,a_n-1)}$ and $x_n M_{(a_0,a_1,\ldots,a_n-1)}\subset in(J_\mathbf{a})$ by Remark \ref{vars_inJa} (2). So, this is also a contradiction.

Now, consider another condition $\LT(r)x_{n-1}^{m_{n-1}}\in in(J_\mathbf{a})=\langle in(J_{(a_0,a_1,\ldots,a_{n-1})})\cup x_n M_{(a_0,a_1,\ldots,a_n-1)}\rangle$. Since $\LT(r)x_{n-1}^{m_{n-1}}$ does not concern any $x_n$ in it, we have \\$\LT(r)x_{n-1}^{m_{n-1}}\in in(J_{(a_0,a_1,\ldots,a_{n-1})})$. Note that this is exactly the same situation as the initial condition above just with less number of variables. If $a_{n-1}$ is odd, then $\LT(r)\in in(J_{(a_0,a_1,\ldots,a_{n-1})})\subset in(J_\mathbf{a})$, a contradiction. And if $a_{n-1}$ is even, then $\LT(r)x_{n-1}^{m_{n-1}}\in in(J_{(a_0,a_1,\ldots,a_{n-2})})$ or $\LT(r)x_{n-1}^{m_{n-1}}\in x_{n-1}\cdot in(J_{(a_0,a_1,\ldots,a_{n-1}-1)})$. Then, similarly we can see that the former leads to a contradiction and the latter implies $\gamma_{n-1}=0$. So, using a condition $\LT(r)x_{n-2}^{m_{n-2}}\in in(J_\mathbf{a})=\langle in(J_{(a_0,a_1,\ldots,a_{n-2})})\cup x_{n-1} M_{(a_0,a_1,\ldots,a_{n-1}-1)}\cup x_n M_{(a_0,a_1,\ldots,a_n-1)}\rangle$, we reach $\LT(r)x_{n-2}^{m_{n-2}}\in in(J_{(a_0,a_1,\ldots,a_{n-2})})$, and so on. Thus, one can repeat the same arguments until the moment that we first meet an odd $a_i$ at some index $i>0$, where the proof can be done by contradiction. 

Hence, the only remaining case is that $a_i$ are all even for $i>0$. Then by repeating the process above, we get $\LT(r)x_1^{m_1}\in (x_0^{a_0+a_1-1}x_1)= x_1 M_{(a_0,a_1-1)}$, which implies $\LT(r)\in M_{(a_0,a_1-1)}$ since $a_1-1$ is odd. Similarly we have $\gamma_1=0$ (otherwise $\LT(r)\in x_1 M_{(a_0,a_1-1)}\subset in(J_\mathbf{a})$ which is a contradiction). But, this situation contradicts to the last condition $\LT(r)x_{0}^{m_{0}}\in in(J_\mathbf{a})$ because $\LT(r)x_{0}^{m_{0}}$ is just a power of $x_0$ and $in(J_\mathbf{a})=\langle(x_0^{a_0+a_1-1}x_1) \cup \bigcup_{i=2}^{n} x_i\cdot M_{(a_0,\ldots,a_{i}-1)}\rangle$ in this case by Remark \ref{vars_inJa} (2).
\end{proof}

Now, we are ready to prove Theorem \ref{main_thm}.\\

\noindent\textbf{Proof of Theorem \ref{main_thm}}
Let us denote the given monomial $X_0^{a_0}X_1^{a_1}\ldots X_n^{a_n}$ by $\mathbf{X}^{\mathbf{a}}$ for a sequence of positive integers $\mathbf{a}=(a_0,\ldots,a_n)$. Since Waring rank of a monomial does not change by relabeling the variables, we may assume that $a_0\ge\ldots\ge a_n$. 

Then, by Proposition \ref{satu}, $J_\mathbf{a}(t)$ is saturated for all but finitely many real number $t$. Take such a nonzero real number $t_0$. By Proposition \ref{Ja} we know that $J_\mathbf{a}(t_0)$ is contained in $(\mathbf{X}^\mathbf{a})^\perp$ and defines $\frac{1}{2}(\prod_{i=0}^n(a_i+1)-\prod_{i=0}^n(a_i-1))$-many distinct real points as a set (let us call this set $\mathbb{X}$). 

Since $J_\mathbf{a}(t_0)$ is a $0$-dimensional ideal and $\deg(J_\mathbf{a}(t_0))=\deg(in(J_\mathbf{a}(t_0)))$ is same as $|\mathbb{X}|$ by Lemma \ref{degree of monomial} and Proposition \ref{degree of ideal}, we have $\HF(T/J_\mathbf{a}(t_0),d)=|\mathbb{X}|$ for all $d\gg0$, which by Proposition \ref{Prop_FromSatToRad} implies that $J_\mathbf{a}(t_0)$ is equal to the defining ideal $I_{\mathbb{X}}$ (in other words, $\mathbb{X}$ is an real apolar set for the monomial $\mathbf{X}^{\mathbf{a}}$).

Thus, by the Apolarity Lemma (Lemma \ref{real apolarity}) we can conclude that  
$$\rank_\R(\mathbf{X}^{\mathbf{a}})\leq |\mathbb{X}|=\frac{1}{2}\left\{\prod_{i=0}^n(a_i+1)-\prod_{i=0}^n(a_i-1)\right\}~~.$$
\qed

\subsection{Remarks and Computations}\label{subsect_3-4}
\medskip

\begin{remark}\label{better bound}
There are some remarks on Theorem \ref{main_thm}.
\begin{enumerate}
\item First of all, we would like to mention that the upper bound in Theorem \ref{main_thm} is sharp in some cases. For the binary case, $\mathbf{a}=(a_0,a_1)$, the real rank of $X_0^{a_0}X_1^{a_1}$ is $a_0+a_1$ in \cite{BCG}. Our bound also shows that 
$$\rank_{\mathbb{R}}(X_0^{a_0}X_1^{a_1})\le\frac{1}{2}((a_0+1)(a_1+1)-(a_0-1)(a_1-1))=a_0+a_1~,$$
which is tight. Further, by \cite[theorem 3.5]{CKOV} it is known that $\rank_\R(\mathbf{X}^{\mathbf{a}})=\rank_\mathbb{C}(\mathbf{X}^{\mathbf{a}})=\frac{1}{2}\prod_{i=0}^n(a_i+1)$ whenever $\min(a_i)=1$. In this case, the bound in Theorem \ref{main_thm} also coincides with this real rank as $\prod_{i=0}^n (a_i-1)=0$, 
\begin{align*}
\rank_{\mathbb{R}}(X_0^{a_0}\cdots X_{n-1}^{a_{n-1}}X_n)&\le\frac{1}{2}\bigg(\prod_{i=0}^n (a_i+1)-\prod_{i=0}^n (a_i-1)\bigg)=\frac{1}{2}\prod_{i=0}^n(a_i+1)~.
\end{align*}
Finally, we would like to point out that the upper bound in Theorem \ref{main_thm} for a special monomial $(x_0\cdots x_n)^2$ is $\frac{1}{2}(3^{n+1}-1)$, which is same as the bound for this monomial given in \cite[proposition 3.6]{CKOV}.

\item In general, for a sequence $\mathbf{a}=(a_0,a_1,\ldots,a_n)$ with minimum $a_n$, the upper bound in Theorem \ref{main_thm} (call this $\textsc{UB}_{\textsf{HM}}$) is better than the bound $\prod_{i=0}^{n-1}(a_i+a_n)$ in \cite[theorem 3.1]{CKOV} (denote by $\textsc{UB}_{\textsf{CKOV}}$). For, using the elementary symmetric polynomials $e_k(a_0,a_1,\ldots,a_{n-1}):=\sum_{0\leq j_0\leq\cdots\leq j_k\leq n-1}a_{j_0}\cdots a_{j_k}$, we see that 
\begin{align*}
\textsc{UB}_{\textsf{CKOV}} &=(a_0+a_n)(a_1+a_n)\cdots(a_{n-1}+a_n)\\
&=e_n(a_0,a_1,\ldots,a_{n-1})+e_{n-1}(a_0,a_1,\ldots,a_{n-1})a_n+e_{n-2}(a_0,a_1,\ldots,a_{n-1})a_n^2+\cdots\\
&=\sum_{i=0}^n e_{n-i}(a_0,a_1,\ldots,a_{n-1})a_n^i~,
\end{align*}
while $\textsc{UB}_{\textsf{HM}}$ can be represented as
$\ds\sum_{i=0}^n e_{n-i}(a_0,a_1,\ldots,a_{n-1})a_n^{i\%2}$ ($i\%2$ is the remainder in the division of $i$ by $2$).

\item Another bound for the real rank can be obtained from the result \cite[corollary 9]{BT15} which says that the maximum rank is bounded above by the twice of the minimum typical rank. For the given sequence $\mathbf{a}$, let $d=\sum_{i=0}^n a_i$. Then the bound is given by $2\cdot \lceil\frac{1}{n+1}\binom{n+d}{n}\rceil$ (say $\textsc{UB}_{\textsf{BT}}$). Let us compare $\textsc{UB}_{\textsf{HM}}$ and $\textsc{UB}_{\textsf{BT}}$ for the case that all the exponents $a_i$ are equal in a fixed total degree $d$. Since $a_0=a_1=\cdots=a_n=k$ and $d=(n+1)k$, $\textsc{UB}_{\textsf{HM}}$ can estimated by 
$$\frac{1}{2}((k+1)^{n+1}-(k-1)^{n+1})=(n+1)\cdot k^n+\binom{n+1}{n-2}k^{n-2}+\cdots\approx (n+1)\cdot k^n~.$$
On the other hand, $\textsc{UB}_{\textsf{BT}}$ is given by
$$2\cdot \lceil\frac{1}{n+1}\binom{(n+1)k+n}{n}\rceil=2\cdot \lceil \frac{1}{n+1}\frac{((n+1)k+n)((n+1)k+n-1)\cdots((n+1)k+1)}{n!}\rceil~,$$
which can be approximated to $2\frac{(n+1)^{n}}{(n+1)!}k^n$.
Hence, $\textsc{UB}_{\textsf{HM}}$ is asymptotically much better than $\textsc{UB}_{\textsf{BT}}$ for this case.
\end{enumerate} 

\begin{table}[htb]
\small{\begin{tabular}{ l r r r r r r r r}
\toprule
			& $\mathbf{a}=(3,3,3)$& (4,3,3) 	& (4,4,3) 	& (4,4,4) 	& (5,5,5,5) 	& (7,7,7,7,7) & (10,9,8,7,6,5,4) & (7,7,7,7,7,7,7)\\
			\midrule
$\textsc{UB}_{\textsf{BT}}$ 		&38		&44		&52		&62 	    &886			&32902 & 8282766 & 8282766\\
\hline
$\textsc{UB}_{\textsf{CKOV}}$ 	& 36		& 42		& 49		& 64		&1000		& 38416 & 2162160 & 7529536\\
\hline
$\textsc{UB}_{\textsf{HM}}$ 	& 28		& 34		& 41		& 49		& 520		& 12496 & 740880 & 908608\\
\bottomrule
\end{tabular}}\\[0.25\baselineskip]
\caption{Comparison of upper bounds for real rank of monomials}\label{compareTable}
\end{table}
\end{remark}

\begin{remark}[Waring rank and decomposition over $\mathbb{Q}$]\label{Qrank} We would like to mention that the bound in Theorem \ref{main_thm} and the procedure for finding a suitable apolar set also hold for the Waring problem over the \textit{rational} numbers. In fact, all the results on the ideal $J_\mathbf{a}(t)$ in the paper, which are devised for producing an apolar set for the upper bound, do hold at least for all but finitely many real number $t$. Thus, for some rational choice of $t$, every necessary statement still remains true (for instance, we can find an apolar set of rational points by taking an appropriate $t\in\mathbb{Q}$ in Proposition \ref{Ja} as in Example \ref{Ex333} below). Therefore, we also have for any monomial $\mathbf{X}^{\mathbf{a}}=X_0^{a_0}X_1^{a_1}\ldots X_n^{a_n}$ with each $a_i>0$,
\begin{equation}\label{Our_bd_rank_over_QQ}
\rank_\mathbb{Q}(\mathbf{X}^{\mathbf{a}})\leq \frac{1}{2}(\prod_{i=0}^n(a_i+1)-\prod_{i=0}^n(a_i-1))~.
\end{equation}
\end{remark}

\begin{example}\label{Ex333}
For the degree $9$ ternary monomial $F=X^3Y^3Z^3$ (i.e. the case of $\mathbf{a}=(3,3,3)$), as looking at the zeros of $J_\mathbf{a}(t)$ when $t=2$, we have the following real (in fact, rational) Waring decomposition of $F$ with $28$ linear forms.
$$-725760 X^3Y^3Z^3=252[(X-Y+Z)^9+(X+Y-Z)^9-(X-Y-Z)^9-(X+Y+Z)^9]$$
$$+2[(X-Y+2Z)^9+(X+Y-2Z)^9-(X-Y-2Z)^9-(X+Y+2Z)^9]$$
$$+2[(X-2Y+Z)^9+(X+2Y-Z)^9-(X-2Y-Z)^9-(X+2Y+Z)^9]$$
$$+2[(2X-Y+Z)^9+(2X+Y-Z)^9-(2X-Y-Z)^9-(2X+Y+Z)^9]$$
$$-(X-2Y+2Z)^9-(X+2Y-2Z)^9+(X-2Y-2Z)^9+(X+2Y+2Z)^9$$
$$-(2X-Y+2Z)^9-(2X+Y-2Z)^9+(2X-Y-2Z)^9+(2X+Y+2Z)^9$$
$$-(2X-2Y+Z)^9-(2X+2Y-Z)^9+(2X-2Y-Z)^9+(2X+2Y+Z)^9~,$$which shows $\rank_\mathbb{R}(X^3Y^3Z^3)\le \frac{1}{2}(4^3-2^3)=28$ (and also $\rank_\mathbb{Q}(X^3Y^3Z^3)\le28$).
\end{example}

When the least exponent is equal to $1$, we determine the rational Waring rank of a monomial.
\begin{corollary}\label{Qrank_a_n=1}
Let $\mathbf{X}^{\mathbf{a}}=X_0^{a_0}X_1^{a_1}\ldots X_n^{a_n}$ with each $a_i>0$ be any monomial in $\QQ[X_0,X_1,\ldots,X_n]$. If $\min(a_i)=1$, then $\rank_\QQ(\mathbf{X}^{\mathbf{a}})=\frac{1}{2}\prod_{i=0}^n(a_i+1)$.
\end{corollary}
\begin{proof}
First, note that, when $\min(a_i)=1$, $$\rank_\QQ(\mathbf{X}^{\mathbf{a}})\ge\rank_\RR(\mathbf{X}^{\mathbf{a}})=\frac{1}{2}\prod_{i=0}^n(a_i+1),$$ where the real rank is determined by \cite[theorem 3.5]{CKOV}. So, the assertion comes by establishing $\rank_\QQ(\mathbf{X}^{\mathbf{a}})\le \frac{1}{2}\prod_{i=0}^n(a_i+1)$ by (\ref{Our_bd_rank_over_QQ}).
\end{proof}

We present an implementation for computing a real apolar set and a real Waring decomposition of any given monomial via the method developed in this article as Macaulay2 code and we execute it for some cases.
\medskip

{\footnotesize
\begin{Verbatim}[commandchars=&!$]
Macaulay2, version 1.16
with packages: &colore!airforceblue$!ConwayPolynomials$, &colore!airforceblue$!Elimination$, &colore!airforceblue$!IntegralClosure$, &colore!airforceblue$!InverseSystems$, 
               &colore!airforceblue$!LLLBases$, &colore!airforceblue$!PrimaryDecomposition$, &colore!airforceblue$!ReesAlgebra$, &colore!airforceblue$!TangentCone$

&colore!magenta$!-- define functions to generate the polynomials F and G, the ideal J, and the points of V(J)$
&colore!blue$!i1 :$ fun:=(R,knd,ind,jnd,t)->(R_ind^2-t^(2*knd)*R_jnd^2);
&colore!blue$!i2 :$ F:=(R,a,ind,jnd,t)->(
		&colore!bleudefrance$!product$ &colore!purple$!for$ l &colore!purple$!from$ -&colore!bleudefrance$!floor$(a_ind/4)-&colore!bleudefrance$!floor$(a_jnd/4+1/2) 
		&colore!purple$!to$ &colore!bleudefrance$!floor$(a_ind/4+1/2)+&colore!bleudefrance$!floor$(a_jnd/4) &colore!purple$!list$ fun(R,l,ind,jnd,t));
&colore!blue$!i3 :$ G:=(R,a,ind,jnd,t)->(
		a':=&colore!purple$!for$ i &colore!purple$!from$ 0 &colore!purple$!to$ #a-1 &colore!purple$!list$ 2*&colore!bleudefrance$!floor$((a_i-1)/2)+1;
		R_ind^((a_ind+1)%2)*R_jnd^((a_jnd+1)%2)*F(R,a',ind,jnd,t));
&colore!blue$!i4 :$ genJ:=(R,a,t)->(
		&colore!bleudefrance$!ideal flatten$ &colore!purple$!for$ i &colore!purple$!from$ 0 &colore!purple$!to$ #a-1 &colore!purple$!list$ &colore!purple$!for$ j &colore!purple$!from$ i+1 &colore!purple$!to$ #a-1 &colore!purple$!list$ G(R,a,i,j,t));
&colore!blue$!i5 :$ ptsgen:=(R,a,epos,t)->(
	emems:=&colore!bleudefrance$!positions$(a,&colore!bleudefrance$!even$);
	&colore!purple$!if$(#emems==0) &colore!purple$!then$ (
		bound:=&colore!purple$!for$ i &colore!purple$!from$ 0 &colore!purple$!to$ #a-1 &colore!purple$!list$ &colore!purple$!for$ j &colore!purple$!from$ -&colore!bleudefrance$!floor$(a_i/4) &colore!purple$!to$ &colore!bleudefrance$!floor$((a_i+2)/4) &colore!purple$!list$ j;
		ilist:=(&colore!purple$!for$ i &colore!purple$!from$ 0 &colore!purple$!to$ #bound-1 &colore!purple$!list$ 0)..(&colore!purple$!for$ i &colore!purple$!from$ 0 &colore!purple$!to$ #bound-1 &colore!purple$!list$ #bound_i-1);
		pospts:=&colore!purple$!for$ i &colore!purple$!from$ 0 &colore!purple$!to$ #ilist-1 &colore!purple$!list$ &colore!purple$!if$(&colore!bleudefrance$!product$ (ilist_i)==0) &colore!purple$!then$ 
			&colore!purple$!for$ j &colore!purple$!from$ 0 &colore!purple$!to$ #bound-1 &colore!purple$!list$ t^(bound_j_(ilist_i_j)) &colore!purple$!else$ continue;
		pm:=(&colore!purple$!for$ i &colore!purple$!from$ 0 &colore!purple$!to$ #a-2 &colore!purple$!list$ 0)..(&colore!purple$!for$ i &colore!purple$!from$ 0 &colore!purple$!to$ #a-2 &colore!purple$!list$ 1);
		pmlist:=&colore!purple$!for$ i &colore!purple$!from$ 0 &colore!purple$!to$ #pm-1 &colore!purple$!list$ &colore!purple$!for$ j &colore!purple$!from$ 0 &colore!purple$!to$ #pm_i-1 &colore!purple$!list$ (-1)^(pm_i_j);
		pts:=&colore!bleudefrance$!flatten$ &colore!purple$!for$ k &colore!purple$!from$ 0 &colore!purple$!to$ #pospts-1 &colore!purple$!list$ &colore!purple$!for$ i &colore!purple$!from$ 0 &colore!purple$!to$ #pmlist-1 &colore!purple$!list$(
		point:=(&colore!purple$!for$ j &colore!purple$!from$ 0 &colore!purple$!to$ #pospts_k-1 &colore!purple$!list$ 
			&colore!purple$!if$(&colore!purple$!not$ j==#pospts_k-1) &colore!purple$!then$ pospts_k_j*pmlist_i_j &colore!purple$!else$ pospts_k_j);
		&colore!purple$!if$(&colore!purple$!not$ #epos==0) &colore!purple$!then$ &colore!purple$!for$ i &colore!purple$!from$ 0 &colore!purple$!to$ #epos-1 do point=&colore!bleudefrance$!insert$(epos_i,0,point);
		point
		)
		)		
		&colore!purple$!else$(
			temp:=emems_(#emems-1);
			a':=&colore!bleudefrance$!drop$(a,{temp,temp});
			a'':=&colore!purple$!for$ i &colore!purple$!from$ 0 &colore!purple$!to$ #a-1 &colore!purple$!list$ &colore!purple$!if$(i==temp) &colore!purple$!then$ a_i-1 &colore!purple$!else$ a_i;
			ptsgen(R,a',{temp}|epos,t)|ptsgen(R,a'',epos,t)
		));
&colore!magenta$!-- the case of (3,3,3), take t=2 --$		
&colore!blue$!i6 :$ a={3,3,3},t=2, n=#a-1;
&colore!blue$!i7 :$ R = &colore!darkspringgreen$!QQ$[x_0..x_n];
&colore!blue$!i8 :$ J=genJ(R,a,t)
&colore!black$!o8 =$ &colore!bleudefrance$!ideal$(x_0^6-(21/4)*x_0^4*x_1^2+(21/4)*x_0^2*x_1^4-x_1^6,x_0^6-(21/4)*x_0^
      4*x_2^2+(21/4)*x_0^2*x_2^4-x_2^6,x_1^6-(21/4)*x_1^4*x_2^2+(21/4)*x_1^2*x_
      2^4-x_2^6)
&colore!black$!o8 :$ &colore!darkspringgreen$!Ideal$ of R
&colore!magenta$!-- list the points of our apolar set from the ideal J$	
&colore!blue$!i9 :$ ptlist=ptsgen(R,a,{},t)
&colore!black$!o9 =$ {{1, 1, 1}, {1, -1, 1}, {-1, 1, 1}, {-1, -1, 1}, {1, 1, 2}, {1, -1, 2},
      {-1, 1, 2}, {-1, -1, 2}, {1, 2, 1}, {1, -2, 1}, {-1, 2, 1}, {-1, -2, 1},
      {1, 2, 2}, {1, -2, 2}, {-1, 2, 2}, {-1, -2, 2}, {2, 1, 1}, {2, -1, 1},
      {-2, 1, 1}, {-2, -1, 1}, {2, 1, 2}, {2, -1, 2}, {-2, 1, 2}, {-2, -1, 2},
      {2, 2, 1}, {2, -2, 1}, {-2, 2, 1}, {-2, -2, 1}}
&colore!black$!o9 :$ &colore!darkspringgreen$!List$
&colore!blue$!i10 :$ r=#ptlist
&colore!black$!o10 =$ 28
&colore!blue$!i11 :$ C=&colore!darkspringgreen$!QQ$[c_0..c_(r-1)];
&colore!blue$!i12 :$ S=C[X_0..X_n];
&colore!magenta$!-- produce the linear forms L_i in the Waring decomposition \sum_i \lambda_i*L_i$
&colore!blue$!i13 :$ linlist=&colore!purple$!for$ i &colore!purple$!from$ 0 &colore!purple$!to$ r-1 &colore!purple$!list$ &colore!bleudefrance$!sum$ &colore!purple$!for$ j &colore!purple$!from$ 0 &colore!purple$!to$ n &colore!purple$!list$ ptlist_i_j*S_j;
&colore!blue$!i14 :$ eq=&colore!bleudefrance$!sum$ &colore!purple$!for$ i &colore!purple$!from$ 0 &colore!purple$!to$ r-1 &colore!purple$!list$ c_i*linlist_i^(&colore!bleudefrance$!sum$ a);
&colore!blue$!i15 :$ L=&colore!bleudefrance$!sub$(&colore!bleudefrance$!ideal$((&colore!bleudefrance$!coefficients$(eq-&colore!bleudefrance$!product$ &colore!purple$!for$ i &colore!purple$!from$ 0 &colore!purple$!to$ n &colore!purple$!list$ S_i^(a_i)))_1),C);
&colore!black$!o15 :$ &colore!darkspringgreen$!Ideal$ &colore!purple$!of$ C
&colore!blue$!i16 :$ gL=&colore!bleudefrance$!gens gb$ L;
&colore!black$!o16 :$ &colore!darkspringgreen$!Matrix$
&colore!magenta$!-- compute the coefficients \lambda_i in the Waring decomposition \sum_i \lambda_i*L_i$
&colore!blue$!i17 :$ clist=&colore!bleudefrance$!reverse$ &colore!purple$!for$ i &colore!purple$!from$ 0 &colore!purple$!to$ r-1 &colore!purple$!list$ -&colore!bleudefrance$!coefficient$(c_0^0,gL_i_0)/&colore!bleudefrance$!leadCoefficient$(gL_i_0)
&colore!black$!o17 =$ {1/2880, -1/2880, -1/2880, 1/2880, 1/362880, -1/362880, -1/362880,
      1/362880, 1/362880, -1/362880, -1/362880, 1/362880, -1/725760, 1/725760,
      1/725760, -1/725760, 1/362880, -1/362880, -1/362880, 1/362880, -1/725760,
      1/725760, 1/725760, -1/725760, -1/725760, 1/725760, 1/725760, -1/725760}
&colore!black$!o17 :$ &colore!darkspringgreen$!List$
&colore!magenta$!-- verify the monomial by retrieving it from the Waring decomposition \sum_i \lambda_i*L_i$
&colore!blue$!i18 :$ &colore!bleudefrance$!sum$ &colore!purple$!for$ i &colore!purple$!from$ 0 &colore!purple$!to$ r-1 &colore!purple$!list$ clist_i*linlist_i^(&colore!bleudefrance$!sum$ a)
&colore!black$!o18 =$ X_0^3*X_1^3*X_2^3
&colore!black$!o18 :$ S
&colore!magenta$!-- the case of (4,4,4,4), take t=3 in this time--$	
&colore!blue$!i19 :$ a={4,4,4,4},t=3, n=#a-1;
&colore!blue$!i20 :$ R = &colore!darkspringgreen$!QQ$[x_0..x_n];
&colore!blue$!i21 :$ J=genJ(R,a,t);
&colore!blue$!i22 :$ ptlist=ptsgen(R,a,{},t);
&colore!blue$!i23 :$ r=#ptlist
&colore!black$!o23 =$ 272
&colore!blue$!i24 :$ C=&colore!darkspringgreen$!QQ$[c_0..c_(r-1)];
&colore!blue$!i25 :$ S=C[X_0..X_n];
&colore!blue$!i26 :$ linlist=&colore!purple$!for$ i &colore!purple$!from$ 0 &colore!purple$!to$ r-1 &colore!purple$!list$ &colore!bleudefrance$!sum$ &colore!purple$!for$ j &colore!purple$!from$ 0 &colore!purple$!to$ n &colore!purple$!list$ ptlist_i_j*S_j;
&colore!blue$!i27 :$ eq=&colore!bleudefrance$!sum$ &colore!purple$!for$ i &colore!purple$!from$ 0 &colore!purple$!to$ r-1 &colore!purple$!list$ c_i*linlist_i^(&colore!bleudefrance$!sum$ a);
&colore!blue$!i28 :$ L=&colore!bleudefrance$!sub$(&colore!bleudefrance$!ideal$((&colore!bleudefrance$!coefficients$(eq-&colore!bleudefrance$!product$ &colore!purple$!for$ i &colore!purple$!from$ 0 &colore!purple$!to$ n &colore!purple$!list$ S_i^(a_i)))_1),C);
&colore!black$!o28 :$ &colore!darkspringgreen$!Ideal$ &colore!purple$!of$ C
&colore!blue$!i29 :$ gL=&colore!bleudefrance$!gens gb$ L;
&colore!black$!o29 :$ &colore!darkspringgreen$!Matrix$
&colore!blue$!i30 :$ clist=&colore!bleudefrance$!reverse$ &colore!purple$!for$ i &colore!purple$!from$ 0 &colore!purple$!to$ r-1 &colore!purple$!list$ -&colore!bleudefrance$!coefficient$(c_0^0,gL_i_0)/&colore!bleudefrance$!leadCoefficient$(gL_i_0);
&colore!blue$!i31 :$ &colore!bleudefrance$!sum$ &colore!purple$!for$ i &colore!purple$!from$ 0 &colore!purple$!to$ r-1 &colore!purple$!list$ clist_i*linlist_i^(&colore!bleudefrance$!sum$ a)
&colore!black$!o32 =$ X_0^4*X_1^4*X_2^4*X_3^4
&colore!black$!o18 :$ S

\end{Verbatim}
} \noindent




\end{document}